\documentclass[12 pt]{amsart}
\usepackage{geometry}
\usepackage{graphicx}

\usepackage{hyperref}
\usepackage{enumerate}
\usepackage{xcolor}
\usepackage{amsmath}
\usepackage{setspace}
\usepackage{subfiles}
\usepackage{accents}
\usepackage{mathtools}

\renewcommand{\baselinestretch}{1.20}%\normalsize

\usepackage{lipsum}
\usepackage[colorinlistoftodos,textsize=scriptsize]{todonotes}
\usepackage{marginnote}
\usepackage{zref-savepos}

\newcount\todocount

\newcommand{\checkxpos}[3][]{%
  \ifdim \zposx{#2}sp < 20000000sp%
    \mynote[#1]{#3}%
  \else%
    \note[#1]{#3}%
  \fi%
}

\newcommand{\mytodo}[2][]{%
  \zsaveposx{todo\the\todocount}%
  \checkxpos[#1]{todo\the\todocount}{#2}%
  \global\advance\todocount1\relax
}

\newcommand{\mynote}[2][]{{%
  \let\marginpar\marginnote
  \reversemarginpar
  \renewcommand{\baselinestretch}{0.8}%
  \todo[#1]{#2}}}

\newcommand{\note}[2][]{\renewcommand{\baselinestretch}{0.8}\todo[#1]{#2}}
\newcommand{\Real}{\mathbb{R}}

\newcommand{\pd}[2]{\partial_{#2}{#1}}

\newcommand{\gra}{\nabla}

\newcommand{\RR}{\mathbb{R}}

\newcommand{\ph}{\phi}

\newcommand{\Op}{\Omega_{p}}

\newcommand{\Om}{\Omega}

\newcommand{\wt}[1]{\widetilde{#1}}
\newcommand{\bbar}[1]{\overline{#1}}
\newcommand{\ii}[2]{\int_{#1}^{#2}}
\newcommand{\ssu}[1]{\underset{#1}{\sup}\,}
\newcommand{\inff}[1]{\underset{#1}{\inf}\,}

\newcommand{\xx}[1]{x_{#1}}

\newcommand{\mfc}{\mathfrak{c}}
\newcommand{\mfm}{\mathfrak{m}}
\newcommand{\mff}{\mathfrak{f}}

\newtheorem{theorem}{Theorem}[section]
\newtheorem{lemma}[theorem]{Lemma}
\newtheorem{proposition}[theorem]{Proposition}
\newtheorem{corollary}[theorem]{Corollary}

\newtheorem{assumption}{Assumption}[section]

\newtheorem{remark}{Remark}[section]
\numberwithin{equation}{section}

\begin{document}

\title[Steady solutions of hydrostatic Euler equations]{On the characterization, existence and uniqueness of steady solutions to the hydrostatic Euler equations in a nozzle}


%    author one information
\author{Wang Shing Leung}
\address{Wang Shing Leung, Department of Mathematics, Temple University, 1805 N. Broad Street, Philadelphia, PA 19122, USA}
%\curraddr{Wang Shing Leung, Department of Mathematics, Temple University, 1805 N. Broad Street, Philadelphia, PA 19122, USA}
\email{wang.shing.leung@temple.edu}
%\thanks{}

%    author two information
\author{Tak Kwong Wong}
\address{Tak Kwong Wong, Department of mathematics, The University of Hong Kong, Pokfulam, Hong Kong}
%\curraddr{}
\email{takkwong@maths.hku.hk}
%\thanks{}

%    author three information
\author{Chunjing Xie}
\address{Chunjing Xie, School of mathematical Sciences, Institute of Natural Sciences, Ministry of Education Key Laboratory of Scientific and Engineering Computing, and CMA-Shanghai, Shanghai Jiao Tong University, 800 Dongchuan Road, Shanghai, China}
%\curraddr{}
\email{cjxie@sjtu.edu.cn}
%\thanks{1111}

%    \subjclass is required.
%\subjclass[2010]{Primary }

\dedicatory{}

\date{\today}
%    Abstract is required.

\begin{abstract}
	Incompressible Euler flows in narrow domains, in which the horizontal length scale is much larger than other scales, play an important role in applications, and their leading-order behavior can be described by the hydrostatic Euler equations. In this paper, we show that steady solutions of the hydrostatic Euler equations in an infinite  strip strictly away from stagnation must be shear flows. Furthermore,  we prove the existence, uniqueness, and asymptotic behavior of global steady solutions to the hydrostatic Euler equations in general nozzles. In terms of stream function formulation, the hydrostatic Euler equations can be written as a degenerate elliptic equation, for which the Liouville type theorem in a strip is a consequence of the analysis for the second order ordinary differential equation (ODE). The analysis on the associated ODE also helps determine the far field behavior of solutions in general nozzles, which plays an important role in guaranteeing the equivalence of stream function formulation.   One of the key ingredients for the analysis on flows in a general nozzle is a new transformation, which combines a change of variable  and an Euler-Lagrange transformation. With the aid of this new transformation, the solutions in the new coordinates enjoy explicit representations so that the regularity with respect to the horizontal variable can be gained in a clear way.
\end{abstract}


\maketitle

%\tableofcontents

%\setcounter{tocdepth}{1}
%\listoftodos[List of suggested changes]


\section{Introduction and main results}

A fundamental problem in fluid dynamics is to study the behaviors of flows in nozzles. For inviscid flows, there are lots of studies on the well-posedness of the flows in an infinitely long nozzle, for example,  the existence, uniqueness, and asymptotic behaviors of global solutions to the incompressible/compressible Euler equations in infinitely long nozzles were proved in \cite{Xie_Xin_2, Xie_Xin, Xie_Xin_3,Du_Xie_Xin,Chen_Huang_Wang_Xiang, LLX} and references therein. In these studies,  the Liouville type theorem for flows  in a strip, which asserts the uniqueness of solutions, not only helps establish the far field behavior of solutions  but also plays an important role in achieving the existence of solutions with nonzero vorticity in general nozzles.
Recently, a prominent Liouville type theorem for two-dimensional steady incompressible Euler equations  was established in \cite{Hamel_Nadirashvili_1}, where it was proved that any steady inviscid incompressible flow strictly away from stagnation, in a two-dimensional infinitely long strip must be a shear flow.
For more detailed description on the analysis for incompressible/compressible Euler system in a nozzle, we refer the reader to the literature review at the end of this section.

Near the solid boundary, the viscosity should play an important role for the behavior of flows. In 1950s, Leray proposed a problem on the well-posedness for the steady Navier-Stokes system in an infinitely long nozzle with asymptotic flat ends where the flows should tend to Poiseuille flows, the shear flows in straight nozzles. This problem is called Leary problem nowadays \cite{Galdi}, and was proved to be well-posed in \cite{Amick, LS} when the fluxes of flows are small. It is still a challenging open problem to study Leary problem with large flux \cite{Galdi}.  A key issue to solve Leary problem with large fluxes is to establish Liouville type theorem for Navier-Stokes system in a straight cylinder (cf. \cite{WangXie}). 
One of main difficulties is that two-dimensional 
steady Navier-Stokes system is a fourth order equation in terms of the stream function so that many analysis techniques cannot be used effectively.  Even for the simplified model, Prandtl system, the problem is still very difficult. In order to have a better understanding for the Prandtl system in an infinitely long nozzle, here we first study the solutions of its inviscid counterpart \cite{Weinan}, the hydrostatic Euler system, in a strip and in a  general nozzle. We hope that this study can shed some lights on the further study for Prandtl system in a nozzle and Leray problem on Navier-Stokes system in a nozzle.

The two-dimensional hydrostatic Euler equations\footnote{This set of partial differential equations have different names in various literature. For example, it is called the homogeneous hydrostatic equations in \cite{Grenier} and the inviscid Prandtl equations in \cite{Weinan}. In this paper, we adapt the terminology in \cite{Lions_1}.} can be used to describe the leading-order behavior of an ideal flow moving in a narrow domain \cite{Lions_1}, and read
\begin{equation*}\label{equation: time dependent hydrostatic Euler equations}
	\left\{\begin{aligned}
		\pd{v_1}{t} + v_1\pd{v_1}{x_1} + v_2\pd{v_1}{x_2} &= -\pd{p}{x_1},\\
		0 &= -\pd{p}{x_2},\\
		\pd{v_1}{x_1} + \pd{v_2}{x_2} &= 0,
	\end{aligned}\right.
\end{equation*}
where $t$ represents the time, $x := (x_1, x_2)$ represents the spatial coordinates, $v := (v_1, v_2) \in \Real^2$ is the flow velocity, and $p$ corresponds to the scalar pressure of the fluid. There are plenty of studies for hydrostatic Euler equatiosn in the last two decades. For the derivation of the hydrostatic Euler equations, one may refer to \cite{Brenier_1, Brenier_2, Grenier, Lions_1, Masmoudi_Wong} and references therein. The unsteady solutions of the hydrostatic Euler equations have been studied extensively. For example, the local well-possedness of solutions was established in \cite{Brenier_3, Kukavica_Masmoudi_Vicol_Wong, Kukavica_Temam_Vicol_Ziane, Masmoudi_Wong, Strain_Wong} in various settings. On the other hand, the ill-possedness and blow-up of classical solutions were investigated in \cite{Canulef, Cao_Ibrahim_Nakanishi_Titi, HanKwan_Nguyen, Ibrahim_Lin_Titi, Renardy_1, Wong}. 

In this paper, we study the two-dimensional steady hydrostatic Euler equations, which reads
\begin{equation}\label{equation : steady state hydrostatic Euler equation}
	\left\{\begin{aligned}
		v_1\pd{v_1}{x_1} + v_2\pd{v_1}{x_2} &= -\pd{p}{x_1}, \\
		0&=-\pd{p}{x_2},\\
		\pd{v_1}{x_1} + \pd{v_2}{x_2}&=0,
	\end{aligned}\right.
\end{equation}
in a nozzle domain $\Omega$ subject to the slip boundary condition
\begin{equation}\label{boundary condition for the flow in perturbed domain}
	(v_1, v_2)\cdot\nu = 0\quad \text{on}\,\, \partial \Omega,
\end{equation}
where $\nu$ is the unit outward normal vector to the boundary $\partial\Omega$.
The following two problems are considered:
\begin{enumerate}[(i)]
	\item Liouville type theorem for the steady hydrostatic Euler equations in a strip;
	\item the existence, uniqueness, and asymptotic behaviors of global solutions to steady hydrostatic Euler equations in general infinitely long nozzles.
\end{enumerate}

Our first result of this paper, namely the  Liouville type theorem for the steady hydrostatic Euler equations in a strip, can be stated as follows. 
\begin{theorem}[Liouville type theorem]\label{main result}
	Let $\Omega=\Omega_0:=\RR\times(0,1)$. Suppose that $v(x) := (v_1(x), v_2(x)) \in C^2(\overline{\Omega_0})$ satisfies \eqref{equation : steady state hydrostatic Euler equation} with boundary condition \eqref{boundary condition for the flow in perturbed domain}. If $\epsilon_0 := \underset{x \in \Omega_0}{\inf}|v(x)| > 0$, then $v$ is a steady shear flow, namely
	\begin{equation*}
		v(x) = (v_1(x_2), 0)\quad \text{in}\ \overline{\Omega_0}.
	\end{equation*}    
\end{theorem}

We have the following remark on Theorem \ref{main result}.
\begin{remark}
When the condition $\underset{x \in \Omega_0}{\inf}|v(x)| > 0$   in Theorem \ref{main result} violates, there exists a 
 non-shear flow 
\[
v(x)=\nabla^\perp (e^{x_1}\sin(\pi x_2)) = (-\pi e^{x_1}\cos(\pi x_2), e^{x_1}\sin(\pi x_2) ),\quad p =-\frac{\pi^2}{2} e^{2x_1}
\]
for the steady hydrostatic Euler system  \eqref{equation : steady state hydrostatic Euler equation} in the strip $\Omega_0$ with slip boundary condition \eqref{boundary condition for the flow in perturbed domain}, which is also a solution of incompressible Euler system in $\Omega_0$ (\cite{Hamel_Nadirashvili_1}).
\end{remark}
If the domain is a nozzle with general geometry, i.e., $\Omega =\Omega_p$, where 
\begin{equation}\label{definition: definition of Omega_p}
	\Om_p:=\{(\xx{1}, \xx{2}): \xx{1}\in\RR, s_0(\xx{1}) < \xx{2} < s_1(\xx{1})\},
\end{equation}
and its boundaries $\pd{\Om_p}{}:= S_0\cup S_1$ with
\begin{equation}\label{definition: form of the nozzles wall}
	S_i = \{(x_1, x_2) : x_1\in\RR, \xx{2} = s_i(\xx{1})\},\quad i = 0,1.
\end{equation}
The boundary $S_i (i=0, 1):\RR\to\RR$,  is assumed to satisfy the following structural  assumptions.
\begin{assumption}\label{structural assumption on Omega p}
	The boundaries of the nozzle $\Omega_p$ are assumed to satisfy the following structural assumptions.
	\begin{enumerate}
		\item [(A1)] The functions $s_0$ and $s_1$ are bounded $C^2(\RR)$ functions, and  satisfy
		\[s_1(x_1) > s_0(x_1)\quad\text{for all}\,\, x_1\in\RR.\]
		
		\item [(A2)] The functions $s_0$ and $s_1$ satisfy the following upstream behaviors:
		\begin{equation*}
			\lim_{x_1\to-\infty}(s_1(x_1),s_0(x_1)) = (1,0).
		\end{equation*}
		\item [(A3)] At the downstream of the nozzle,  there exists constants $a\in \RR$ and $\sigma>0$ such that
		\begin{equation*}
			\lim_{x_1\to\infty}(s_1(x_1),s_0(x_1)) = (a,a+\sigma).
		\end{equation*}
	\end{enumerate}
\end{assumption}

There are a few remarks on the structural assumptions for the nozzles.

\begin{remark}\label{remark: remark on the boundedness of the height of nozzle}
	With the aid of Barbalat inequality (cf. \cite{Barbalat}), the functions $s_0$ and $s_1$ with finite $C^2(\mathbb{R})$ norm must satisfy
	\begin{equation*}
		s_0', s_1' \to 0 \,\, \text{as}\,\, x_1\to -\infty
	\end{equation*}
	as long as (A2) of Assumption \ref{structural assumption on Omega p} holds.
 Assumption~\ref{structural assumption on Omega p} implies that the nozzle $\Omega_p$ has finite width, i.e., there exist two constants $\underline{d}$, $\bar{d}\in (0, +\infty)$ such that
 \begin{equation}
 	0<\underline{d}=	\inf_{\RR}s(x_1)\leq \sup_{\RR}s(x_1) =\bar{d},
 \end{equation}
where 
\begin{equation}\label{defs}
	s(x_1) := s_1(x_1)-s_0(x_1).
\end{equation}
A special class of nozzles with finite width, which are flat except for a compact subset, satisfy
	\[
	s_0(x_1 )\equiv 0,\quad s_1(x_1) \equiv 1 \quad \text{for}\,\, x\leq -N,
	\]
	and
	\[
	s_0(x_1 )\equiv a,\quad s_1(x_1) \equiv a+\sigma \quad \text{for}\,\, x\geq N.
	\]
\end{remark}

\begin{remark}
It is just for simplicity to assume that the asymptotic heights are $1$ and $\sigma$ at the upstream and downstream, respectively. In fact, one can deal with the nozzles with arbitrary asymptotic heights at far fields.
\end{remark}

\begin{remark}
	In fact, one can study more general case, i.e., the asymptote  of the nozzle boundary are straightlines.  For example, one can remove the boundedness assumption in (A1) and replace (A3) in Assumption \ref{structural assumption on Omega p} by the following conditions:
	\begin{equation}\label{generalnozzle}
	\lim_{x_1\to+\infty} (s_0(x_1)-b_0 -b_1 x_1 )=0\quad \text{and} \quad 	\lim_{x_1\to+\infty} \left(s_1(x_1)-b_0-\sqrt{b_1^2 +1} -b_1 x_1 \right)=0,
	\end{equation}
	\begin{equation}\label{generalnozzlederivatives}
	    \lim_{x_1\to+\infty}s_1'(x_1) = \lim_{x_1\to+\infty}s_0'(x_1) = b_1,
	\end{equation}
	where $b_0$ and $b_1$ are two constants.
\end{remark}

Since there is a hyperbolic mode in steady hydrostatic Euler equations,
in order to completely determine the flows in a general nozzle, we ask the flows to satisfy the asymptotic behavior
\begin{equation}\label{upstreamcond}
	v_1\to v_1^-\quad  \text{uniformly on any compact subset of } (0,1), \quad \text{as}\,\, x_1\to -\infty. 
\end{equation}

The second result of this paper, which concerns the existence, uniqueness, and asymptotic behavior of solutions to the steady hydrostatic Euler equations in a general nozzle $\Omega_p$, can be stated as follows.

\begin{theorem}\label{theorem: existence of solution without sign condition}
	Let $\Omega=\Omega_p$ with $\Omega_p$ defined in \eqref{definition: definition of Omega_p} and satisfy  Assumption~\ref{structural assumption on Omega p}. Suppose that $v_1^-:=v^-_1(x_2) \in C^2([0,1])$ satisfies
	\begin{equation}\label{incomingcond}
		(v^-_1) > 0\ \text{in}\ [0,1]\quad \text{and}\quad 
		(v^-_1)'(0) \leq 0 \leq (v^-_1)'(1).
	\end{equation} 
	Then there exists a classical solution $(v,p)$ to the steady hydrostatic Euler equations \eqref{equation : steady state hydrostatic Euler equation} subject to the boundary condition \eqref{boundary condition for the flow in perturbed domain}, and asymptotic condition \eqref{upstreamcond}, which satisfies
	\begin{equation}\label{e:v1>0inbarOmegap}
		v_1 > 0\quad\text{in}\ \overline{\Omega_p},
	\end{equation}
	and the far field behaviors:
	\begin{equation}\label{upstreambehavior}
		(v_1, v_2)\to (v_1^-, 0)	\quad \text{uniformly in any compact set of } (0,1) \text{ as }x_1\to -\infty
	\end{equation}
and
	\begin{equation}\label{downstreambehavior}
	(v_1, v_2)\to (v_1^+, 0)	\quad \text{uniformly as in any compact set of } (a,a+\sigma) \text{ as  }x_1\to +\infty
\end{equation}
where $v_1^+$ is uniquely determined by $v_1^-$ and $\sigma$ appeared in (A3) of Assumption \ref{structural assumption on Omega p}. More precisely, we know that $v_1^+:=\frac{(\phi^\infty)'}{\sigma}$, where $\phi^\infty$ satisfies \eqref{equation: ODE of phi in downstream}. Furthermore, the solution is also unique in the class of functions satisfying \eqref{e:v1>0inbarOmegap} and \eqref{upstreambehavior}.
\end{theorem}

We have the following remarks on Theorem \ref{theorem: existence of solution without sign condition}.

\begin{remark}
The conditions \eqref{incomingcond} are also needed in the previous studies for steady flows in nozzles, and they play an important role in studying solutions of compressible Euler equations, see \cite{Chen_Huang_Wang_Xiang, Xie_Xin} and references therein.
\end{remark}
\begin{remark}\label{remarkdownstream}
	In addition, if $\Om_p$ satisfies \eqref{generalnozzle}-\eqref{generalnozzlederivatives} instead of (A3) of Assumption \ref{structural assumption on Omega p}, then the flow $v$ also has the following far field behavior:
	\begin{equation}\label{equality: behavior of the flow at downstream for compact case}
		\lim_{x_1\to +\infty} (v_1, v_2) = \left({v_1^+}, {b_1 v_1^+}\right),
	\end{equation}
	where $v_1^+:[0,1]\to(0,\infty)$ is a $C^1$ function determined by $v_1^{-}$, $s_0$, and $s_1$. More precisely, we know that $v_1^+:= \frac{(\phi^{\infty})'}{\sqrt{b_1^2+1}}$ where $\phi^\infty$ satisfies \eqref{equation: ODE of phi in downstream} with $\sigma^2$ replaced by $b_1^2+1
	$.
\end{remark}

Before presenting the key ideas for the proof of main results in this paper, 
in the following, we give a brief summary on known results for inviscid flows in nozzles.

 Because of its rich applications in physics and engineering \cite{Bers_1, CourantF, Dafermos_1, LandauL}, the flows in nozzles provide many interesting problems in both fluid mechanics and mathematical analysis.
The problem on the existence of steady subsonic irrotational solutions for compressible Euler system in an infinitely long nozzle was posed in \cite{Bers_1}. The problem was solved in \cite{Xie_Xin_2,Xie_Xin_3} for irrotational subsonic and subsonic-sonic weak solutions in two-dimensional and three-dimensional axisymmetric nozzles. The existence of multidimensional irrotational subsonic and weak subsonic-sonic solutions were established in \cite{Du_Xin_Yan, Huang_Wang_Wang}. With the help of the detailed analysis for the behavior of flows near sonic points, regular irrotational subsonic-sonic flows were obtained in \cite{WangXin}.

For compressible Euler flows with nonzero vorticity, the existence of unique global subsonic isentropic flows was proved in \cite{Xie_Xin} when the variation of Bernoulli's function in the upstream is suitably small and  the mass flux of the flows is less than a critical value. In \cite{Du_Xie_Xin}, a class of subsonic flows with large vorticity were obtained where the horizontal velocity in the upstream is convex.  Later the subsonic solutions in nozzles were established in \cite{Chen_Huang_Wang_Xiang} when the incoming velocity satisfies \eqref{incomingcond}. There are studies for subsonic flows in periodic nozzles, axisymmetric nozzles, etc, one may refer to \cite{Chen_1, Chen_Du_Xie_Xin, Chen_Xie, Chen_2, Du_Duan, Duan_Luo_2, Duan_Luo_1} and references therein. 

In fact, in order to obtain the existence and asymptotic behavior of solutions for flows in infinitely long nozzles, a Liouville type theorem for solutions in infinitely long strips plays a crucial role.
The Liouville type theorem or characterizations of steady flows, which asserts the rigidity or uniqueness of solutions in a special class of domains, such as infinite strips, annulus, or the whole space, etc.  As it was mentioned at the beginning of this paper, in \cite{Hamel_Nadirashvili_1}, it was proved that any $C^2$ solution strictly away from stagnation, to the steady incompressible Euler equations must be a shear flow in the infinite strip. For the study on Liouville type theorem for steady incompressible Euler system in the whole plane and annulus domain, one may refer to \cite{Hamel_Nadirashvili_3,Hamel_Nadirashvili_2}. 
Recently, the Liouville type theorem for incompressible Euler equations with stagnation points was established in \cite{LLX} where the flows in general nozzles were also studied with the aid of this Liouville type theorem. The Liouville type theorems for incompressible Euler equations and 3D primitive equations were also investigated in \cite{Peralta-Salas_Slobodeanu, Ruiz}.

Here we provide the key ideas for the proof of main results. First, we use the stream function formulation to reduce the hydrostatic Euler equations into a single second order degenerate equation. The Liouville type theorem for the hydrostatic Euler equations is equivalent to the uniqueness of boundary value problem for an ordinary differntial equation (ODE).
For solutions in a general nozzle, when the incoming horizontal velocity satisfies \eqref{incomingcond}, one can use the comparison principle to obtain the bounds and existence of solutions for the degenerate partial differential equation (PDE). However, the regularity of these solutions in the $x_1$-direction is not clear. In order to show that these solutions are indeed classical ones, we exploit the structure of the equation and introduce a new transformation so that we can solve the associated equation with an explicit representation formula. This helps to gain the regularity clearly and shows that solutions are classical ones.

The rest of this paper is organized as follows. In Section~\ref{Preliminary}, we  introduce the stream function formulation to reduce the steady hydrostatic Euler equations into a single second order equation.  In Section~\ref{section:Liouville-type theorem for the hydrostatic Euler equations},  the method developed in \cite{Hamel_Nadirashvili_1} and the analysis for ODE developed in \cite{Berestycki_Nirenberg_1} are adapted  to prove the Liouville type theorem (i.e., Theorem~\ref{main result}) for the steady hydrostatic Euler equations. In Section~\ref{section: Existence of steady state hydrostatic flows in infinite nozzles}, the existence and uniqueness of solutions to the boundary value problem for the stream function are proved, where the key issue is to analyze a degenerate elliptic equation~with Dirichlet boundary condition. This, together with stream function formulation, shows Theorem \ref{theorem: existence of solution without sign condition}.

\section{Stream function formulation}\label{Preliminary}
The aim of this section is to introduce the stream function formulation  so that  the hydrostatic Euler equations can be reduced into a degenerate second order semilinear equation. 

Let $v:=(v_1, v_2)$ be a sufficiently smooth solution to \eqref{equation : steady state hydrostatic Euler equation}. It follows from the incompressibility condition and the simply connectedness of $\Omega_0$ and $\Omega_p$ that there exists a stream function $\varphi$ such that
\begin{equation}\label{negative gradient per of stream function}
	\pd{\varphi}{\xx{1}} = -v_2\quad \text{and}\quad \pd{\varphi}{\xx{2}} = v_1.
\end{equation} 
Furthermore, direct computations show that the solutions of steady hydrostatic Euler system \eqref{equation : steady state hydrostatic Euler equation} satisfy
\begin{equation}\label{equation : conservation of vorticity}
	\gra^\perp\varphi\cdot\gra\omega = v\cdot\gra\omega = 0,
\end{equation}
where
\begin{equation}\label{definition: vorticity of the hydrostatic Euler equations}
	\omega:=\pd{v_1}{x_2}=\partial_{x_2x_2}\varphi
\end{equation}
is the vorticity for the hydrostatic Euler equations.

If the flow does not contain any stagnation point, at which the flow speed is zero, then with the aid of implicit function theorem, equation \eqref{equation : conservation of vorticity} implies that in the neighborhood of each point $(x_1, x_2)$, $\omega=\omega(\varphi)$. Hence the stream function $\varphi$ satisfies the following  partial differential equation:
\begin{equation}\label{equation: degenerated PDE for HEE}
	\pd{\varphi}{x_2 x_2} = f(\varphi)
\end{equation}
in the neighborhood of each point $x$.
The key issue is whether the representation \eqref{equation: degenerated PDE for HEE} holds  in the whole domain for some function $f$. 

\subsection{Stream function formulation for flows strictly away from stagnation}
The aim of this subsection is to show that the representation \eqref{equation: degenerated PDE for HEE} holds in the whole strip when the flow is strictly away from stagnation. In fact, when the flow in a strip is strictly away from stagnation, the stream function formulation for incompressible Euler system has been studied extensively in \cite{Hamel_Nadirashvili_1}.

Consider the flows in a strip $\Omega_0$, let $\sigma:[0,\infty)\to \overline{\Omega_0}$ be the solution of 
\begin{equation}\label{sigmaprime}
	\left\{
	\begin{aligned}
		&	\sigma'(t) = \nabla \varphi(\sigma(t)),\\
		&	\sigma(0) = (0,0).
	\end{aligned}
	\right.
\end{equation}

\begin{remark}\label{remark2.1}
	Since the solutions of hydrostatic Euler equations satisfy the slip boundary condition \eqref{boundary condition for the flow in perturbed domain}, one has $v_2=0$ on $\partial \Omega_0$. If 
\(\underset{x \in \Omega_0}{\inf}|v(x)| > 0\), then 
$\underset{x_1 \in \mathbb{R}}{\inf}|v(x_1, 0)| > 0$. Without loss of generality, we assume that 
\begin{equation}\label{v1positive}
	v_1(x_1, 0)>0\quad \text{for}\,\, x_1 \in \mathbb{R}
\end{equation}
for flows in  strip $\Omega_0$ studied in the rest of paper.
\end{remark}

With the assumption in previous remark, it follows from  \cite[Remark~2.2]{Hamel_Nadirashvili_1} that $\sigma$ is defined in a maximal time interval $[0, \tau]$ for some $\tau \in (0, +\infty)$.

First, one has the following lemma.
\begin{lemma}[\cite{Hamel_Nadirashvili_1}]\label{endpoint_lemma}
Let $\sigma$ satisfy \eqref{sigmaprime} and flow be strictly away from stagnation.  Then the following statements hold.
\begin{enumerate}
	\item[(a)]	The second end point of the trajectory $\Sigma_{(0,0)}: = \sigma([0,\tau])$ lies on the horizontal line $x_2 = 1$, namely
	\begin{equation*}
		\sigma(\tau) = (\xi, 1) \quad \text{for some } \xi\in \RR.
	\end{equation*}
	
\item[(b)]	For any $x$ in $\overline{\Omega}$, the corresponding regular streamline $\Gamma_x$, which is the streamline of the velocity field passing through $x$, must intersect with $\Sigma_{(0,0)} = \sigma([0, \tau])$.
\end{enumerate}
\end{lemma}

Furthermore, we have the following lemma.
\begin{lemma}[\cite{Hamel_Nadirashvili_1}, Lemma~2.6]\label{lemma about the boundedness of u}
	Assume that $\underset{\Omega_0}{\inf}\ |v| > 0$ and $v_2(x) = 0$ on $\pd{\Omega_0}{}$. If \eqref{v1positive} holds, then there exists a constant $\mfm>0$ such that the associated stream function $\varphi$ is bounded in $\Omega_0$ with
	\begin{equation*}
		0 < \varphi <\mfm\quad\text{in}\quad\Omega_0,
	\end{equation*}
	\begin{equation*}
		\varphi|_{x_2 = 0} = 0,\quad\text{and}\quad\varphi|_{x_2 = 1} = \mfm.
	\end{equation*}
\end{lemma}

\subsection{Stream function formulation for flows in general nozzles}
In order to study hydrostatic Euler equations in a general nozzle, we need to show that the stream function $\varphi$ of the flow satisfies the degenerate elliptic equation~\eqref{equation: degenerated PDE for HEE} in $\Omega_p$ globally and  determine the precise form of the function $f$.
More precisely, we follow the idea in \cite{Xie_Xin} and study the steady solution of \eqref{equation : steady state hydrostatic Euler equation} by prescribing the upstream horizontal velocity profile. Note that the stream function in the upstream can be defined as
\begin{equation}\label{equation: relation of stream function and horizontal velocity in upstream}
	\varphi^{-}(x_2) := \int_{0}^{x_2}v_1^-(\tau)\,d\tau.
\end{equation}
Clearly, $ \varphi^- \in [0, \mfc]$ for $x_2\in [0,1]$, where the constant $\mfc$ is defined as
\begin{equation}\label{mass flux constant}
	\mfc := \int_0^1 v_1^-(\tau)\,d\tau.
\end{equation}
It is worth noting that $\varphi^{-}$ is a strictly increasing function of $\xx{2}$, since it is assumed that $v_1^- > 0$. Hence the inverse function of $\varphi^{-}$ is well-defined and we denote it by $\kappa$, i.e.,
\begin{equation}\label{definition: kappa}
	\varphi = \ii{0}{\kappa(\varphi)}(v_1^-)(\tau)\,d\tau.
\end{equation}
Let the function $f$  be defined by 
\begin{equation}\label{deff}
	f(\cdot) := (v^-_1)'(\kappa(\cdot)).
\end{equation}

With the notations defined above, we prove the following lemma that asserts the equivalence between  the hydrostatic Euler equations and  a  second-order degenerate elliptic PDE.
\begin{lemma}\label{lemma: equivalence lemma}
	Let $\Omega_p$ be defined as in \eqref{definition: definition of Omega_p} and satisfy Assumption~\ref{structural assumption on Omega p}. Assume that $v_1^-\in C^2[0,1]$  satisfies
	\[v_1^- > 0.\]
	Let $\mfc$ and $f$ be defined in \eqref{mass flux constant} and \eqref{deff}, respectively.
	Then the following statements hold.
	\begin{enumerate}[(i)]
		\item Let $(v_1, v_2, p)$ be a classical solution to the steady hydrostatic Euler equations \eqref{equation : steady state hydrostatic Euler equation} supplemented with boundary condition \eqref{boundary condition for the flow in perturbed domain}, and satisfy $v_1 > 0$ in $\Op$ and \eqref{upstreamcond}. Then the stream function $\varphi$, which satisfies $\gra\varphi = (-v_2, v_1)$ and $\varphi(0, s_0(0)) = 0$, solves the boundary value problem
		\begin{equation}\label{equation: elliptic PDE in perturbed domain without flattening}
			\begin{cases}
				\pd{\varphi}{x_2 x_2} = f(\varphi)\quad\text{in}\ \Omega_p\\
				\varphi(x_1, s_0(x_1)) = 0,\quad\varphi(x_1, s_1(x_1)) = \mfc.
			\end{cases}
		\end{equation}
		\item Let  $\varphi$ be a unique solution to \eqref{equation: elliptic PDE in perturbed domain without flattening} in $\mathcal{Y}$, which is defined as
		\begin{equation}\label{defintion: function space Y cal}
			\begin{aligned}
				\mathcal{Y} := \big\{\varphi: \varphi, \pd{\varphi}{x_1}, \pd{\varphi}{x_2}, \pd{\varphi}{x_1 x_2},\pd{\varphi}{x_2 x_2}\in C(\bbar{\Om_p}),\\
				\varphi(x_1, s_0(x_1)) = 0, \varphi(x_1, s_1(x_1)) = \mfc\big\}
			\end{aligned}
		\end{equation}
		equipped with the norm
		\begin{equation*}
			||\varphi||_{\mathcal{Y}} = \ssu{\bbar{\Op}}(|\varphi| + |\pd{\varphi}{x_1}| + |\pd{\varphi}{x_2}| + |\pd{\varphi}{x_1 x_2}| + |\pd{\varphi}{x_2 x_2}|).
		\end{equation*}
		If $\pd{\varphi}{x_2} > 0$, then $(v_1, v_2, p)$ satisfies the steady hydrostatic Euler equations \eqref{equation : steady state hydrostatic Euler equation} in $\Omega_p$ supplemented with boundary condition \eqref{boundary condition for the flow in perturbed domain}, where
		\begin{equation}\label{definition: defined flow from stream function}
			v_1 := \pd{\varphi}{x_2},\quad v_2 := -\pd{\varphi}{x_1},\quad p := -\frac{(\pd{\varphi}{x_2})^2}{2} + F(\varphi).
		\end{equation}
	Here, $F$ is a primitive function of $f$.
	\end{enumerate}
\end{lemma}
\begin{proof}
	(i) We first verify the stream function satisfies the boundary conditions in \eqref{equation: elliptic PDE in perturbed domain without flattening}. Let $\gamma$ be a $C^1$ path with initial point $(0, s_0(0))$ and terminal point $(x_1, x_2)\in\Omega_p$. It follows from the incompressibility condition and the fact that $\Omega_p$ is simply connected that there exists a $C^1$-function $\wt{\varphi}$ such
	\[\int_\gamma(-v_2, v_1)\cdot d{\bf s} = \wt{\varphi}(x_1, x_2) - \wt{\varphi}(0, s_0(0)).\]
	Since the stream function is well-defined up to a constant, we define
	\[\varphi(x_1,x_2) := \wt{\varphi}(x_1,x_2) - \wt{\varphi}(0,s_0(0)),\]
	which implies
	\begin{equation}\label{stream function equals to zero at (0, s_0(0))}
		\varphi(0,s_0(0)) = 0.
	\end{equation}
	Hence it follows from \eqref{boundary condition for the flow in perturbed domain}, \eqref{negative gradient per of stream function} and \eqref{stream function equals to zero at (0, s_0(0))} that $\varphi(x_1, s_0(x_1)) = 0$. In addition, with the aid of divergence theorem, incompressibility condition, \eqref{upstreamcond}, \eqref{boundary condition for the flow in perturbed domain}, and \eqref{mass flux constant}, one has
	\[\varphi(0,s_1(0)) = \mfc,\]
	and hence $\varphi(x_1, s_1(x_1)) = \mfc$.
	
	For any points $(x_1, x_2)\in \Op$, there exists a unique streamline that satisfies the system
	\begin{equation}\label{equation: ODE for streamline}
		\begin{cases}
			&\frac{dX_1}{dt}=v_1(X_1(t), X_2(t)),\\
			&\frac{dX_2}{dt}=v_2(X_1(t), X_2(t)),\\
			&(X_1(0), X_2(0))=(x_1, x_2),
		\end{cases}
	\end{equation}
	and the streamline can be defined globally in $\Omega_p$.  As long as $v=(v_1, v_2)\in C^1$, there exists a unique streamline passing through each point. 
	
	In addition, we claim that the streamline through every point inside $\Omega_p$ will never touch $\partial\Omega_p$. Otherwise, suppose that the streamline passing through $(x_1^0, x_2^0)\in\Omega_p$ touches $\partial\Op$ at $(x_1^1, s_0(x_1^1))$. It follows from Green's theorem that one has
	\begin{equation}\label{identity: integral of v1 along vertical line equals 0}
		\int_{s_0(x_1^0)}^{x_2^0} v_1(x_1^0, x_2)\,dx_2 = 0,
	\end{equation}
	which contradicts the assumption that $v_1 > 0$.
	
Clearly,  the stream function is constant along each streamline. On the other hand, it was proved that given that $v_1 > 0$ in $\Omega_p$, for any $(x_1, x_2)\in \Omega_p$,   the streamline passing through $(x_1,x_2)$ has to tend to upstream at $(-\infty, \kappa(\varphi))$ with $\varphi = \varphi(x_1, x_2)$. To prove this claim, we consider the initial value problem  \eqref{equation: ODE for streamline}. It follows directly from the first equation of \eqref{equation: ODE for streamline} and the assumption $v_1 > 0$ in $\Omega_p$ that the streamline cannot stop at any finite region. On the other hand, because the stream function is constant along streamline, Equality~\eqref{equation: relation of stream function and horizontal velocity in upstream} and the assumption $v_1^- > 0$, we can conclude that $X_2(t)\to\kappa(\varphi)$ as $t\to-\infty$. It is worth noting that $\kappa(\varphi)$ is well defined for any $(x_1, x_2) \in \Omega_p$. This is because $\pd{\varphi}{x_2} = v_1 > 0$ in $\Omega_p$, it follows that $0 < \varphi < \mfc$ in $\Omega_p$, and hence $\kappa(\varphi)$ is well-defined for any $(x_1, x_2) \in \Omega_p$. In addition, since the vorticity is also preserved along each streamline,
	one has
	\[
	\lim_{t\to-\infty} \omega(X_1(t), X_2(t))=\omega(X_1(0), X_2(0))= \omega(x_1,x_2).
	\]
	If the flows satisfy far field condition \eqref{upstreamcond}, then one has
	\[
	\omega(x_1, \cdot) =\partial_{x_2} v_1(x_1, \cdot) \to (v_1^-)'(\cdot)\quad  \text{in the sense of distribution as}\,\,x_1\to -\infty. 
	\]
	 Therefore, it follows from the conservation of the vorticity and the property  $X_2(t)\to\kappa(\varphi)$ as $t\to-\infty$ that
	\[\omega(x_1, x_2) = (v_1^-)'(\kappa(\varphi(x_1, x_2)).\]
	%need to fix after meeting
	Furthermore, it follows from \eqref{definition: vorticity of the hydrostatic Euler equations} that $\varphi$ satisfies
	\begin{equation}\label{Equation: second order degenerate PDE for stream function}
		\pd{\varphi}{x_2 x_2} = (v_1^-)'(\kappa(\varphi))=:f(\varphi).
	\end{equation}

	(ii) Suppose that $\varphi$ is a solution of \eqref{equation: elliptic PDE in perturbed domain without flattening}. Then    direct computations yield
	\begin{equation*}
		\pd{v_1}{x_1} + \pd{v_2}{x_2} = \pd{\varphi}{x_2 x_1} - \pd{\varphi}{x_1 x_2} = 0.
	\end{equation*}
Hence the flow defined in \eqref{definition: defined flow from stream function} satisfies the incompressibility condition. Moreover, using Equation~$\eqref{equation: elliptic PDE in perturbed domain without flattening}_1$, we also have
	\begin{equation*}
		-\pd{p}{x_1} = -\pd{\left(-\frac{(\pd{\varphi}{x_2})^2}{2} + F(\varphi)\right)}{x_1} = v_1\pd{v_1}{x_1} + v_2\pd{v_1}{x_2}
	\end{equation*}
	and
	\begin{equation*}
		\pd{p}{x_2} = \pd{\left(\frac{(\pd{\varphi}{x_2})^2}{2} - F(\varphi)\right)}{x_2}
		= \pd{\varphi}{x_2}\pd{\varphi}{x_2x_2} - f(\varphi)\pd{\varphi}{x_2}
		= 0.
	\end{equation*}
Therefore, the flow defined by \eqref{definition: defined flow from stream function} satisfies the steady hydrostatic Euler equations. Lastly, the boundary condition \eqref{boundary condition for the flow in perturbed domain} follows immediately from the fact that $\varphi$ is constant along the boundary and definition \eqref{definition: defined flow from stream function}.
\end{proof}

With the aid of Lemma \ref{lemma: equivalence lemma}, in order to study the existence, uniqueness, and asymptotic behaviors of solutions to the steady hydrostatic Euler equations, we only need to  investigate the boundary value problem \eqref{equation: elliptic PDE in perturbed domain without flattening}.

\begin{remark}
	It is worth noting that the stream function of the steady incompressible Euler equations satisfies
	\begin{equation}\label{equation: uniformly elliptic PDE for IEE}
		\pd{\varphi}{x_1 x_1} + \pd{\varphi}{x_2 x_2} = f(\varphi),
	\end{equation}
	which is a semilinear and uniformly elliptic PDE. Comparing \eqref{equation: degenerated PDE for HEE} with \eqref{equation: uniformly elliptic PDE for IEE},  there is a loss of second-order derivative in the $x_1$-direction. This means that the stream function of the hydrostatic Euler equations satisfies a degenerate elliptic equation \eqref{equation: degenerated PDE for HEE}. Hence the standard estimates for uniformly elliptic equations in \cite{Berestycki_Caffarelli_Nireberg, Gilbarg_Trudinger}, which play an essential role in the proof of \cite{Hamel_Nadirashvili_1}, do not work very well. This is the major difficulty for the analysis on hydrostatic Euler equations.
\end{remark}

 If we regard $x_1$ as a parameter, then equation \eqref{equation: degenerated PDE for HEE} is a nonlinear second-order ordinary differential equation in $x_2$. Thus,  the representation formula for solutions of ODE can be used to obtain several important a priori estimates (cf.  Proposition~\ref{existence of solution}). More precisely, for any fixed $x_1 \in \RR$, using the standard fixed point argument, one can  obtain the existence, uniqueness, and $C^2$-regularity in the $x_2$-direction. If \eqref{incomingcond} is satisfied, then the function $f:\RR\to\RR$, which is defined in \eqref{deff}, is a bounded Lipschitz function satisfying
\begin{equation}\label{assumption: sign condition of f on the boundary}
	f(0) \leq 0\leq f(\mfc).
\end{equation}
This helps to show that $\varphi$ is unique and that $\varphi$ satisfies the boundedness condition $0<\varphi<\mfc$ and $\varphi' > 0$. 

A key issue is to prove the regularity of $\varphi$ in the $x_1$-direction so that the induced flows \eqref{definition: defined flow from stream function} is indeed a classical solution to the steady hydrostatic Euler equations \eqref{equation : steady state hydrostatic Euler equation}. With the aid of a transformation combining a change of variables and Euler-Lagrange transformation, we are able to obtain an explicit representation formula for solutions to some associated partial differential equation (PDE), so that the same regularity of $\varphi$ in the $x_1$-direction can be obtained without any further assumption on $f$. 

\section{Liouville type theorem for flows in a strip}\label{section:Liouville-type theorem for the hydrostatic Euler equations}
In this section, we consider the steady hydrostatic Euler equations in ${\Omega}=\Omega_0 := \RR \times (0,1)$ and prove  Theorem~\ref{main result}, which is a Livouille-type theorem for the hydrostatic Euler equations.

It was showed in~\cite{Hamel_Nadirashvili_1}  that two-dimensional steady incompressible Euler flows in a strip must be a  shear flow, as long as the flow is strictly away from stagnation. The proof in \cite{Hamel_Nadirashvili_1} consists of two main ingredients:
\begin{enumerate}[(i)]
	\item geometric properties of streamlines of the velocity field $v(x)$, which is a consequence of property that flows are strictly away from the stagnation; and
	\item governing equation for the stream function $\varphi$ is a semilinear elliptic equation with Lipschitz nonlinearity. 
\end{enumerate}

For the hydrostatic Euler equations, as what we have done in Section \ref{Preliminary}, when the flows are strictly away from stagnation,  the streamlines are also foliated and  the steady hydrostatic Euler equations can be reduced into a degenerate partial differential equation, which looks like an ordinary differential equation. 

For part (ii), instead of the sliding method used in \cite{Hamel_Nadirashvili_1} for semilinear elliptic equation, we combine the analysis developed in \cite{Berestycki_Nirenberg_1} by Berestycki and Nirenberg for the second order ordinary differential equations to prove the Liouville type theorem for the hydrostatic Euler equations. 

In the following, we prove an analogue of \cite[Theorem~1.4]{Hamel_Nadirashvili_1} for the degenerate equation.
%Do we really need f to be C^1(\RR)%?
\begin{proposition}\label{main_theorem_2}
	Let $\Tilde{\Omega}:=(M,N)\times(0,1)$, where $M,N \in \RR\cup\{\pm\infty\}$, and $\mff:\RR \to \RR$ be a $C^1$ function. Suppose that $\varphi: \overline{\Tilde{\Omega}} \to \RR $ is a $C^2(\overline{\Tilde{\Omega}})$ bounded solution to the equation
	\begin{equation}\label{function relation of hydrostatic euler}
		\pd{\varphi}{x_2 x_2} = \mff(\varphi) 
	\end{equation}
	subject to the boundary conditions
	\begin{equation}\label{boundary condition for u}
		\varphi(x_1, 0) = 0,\quad \varphi(x_1, 1) = \mfm,
	\end{equation}
	where $\mfm>0$ is an arbitrary constant. If the solution $\varphi$ of \eqref{function relation of hydrostatic euler}-\eqref{boundary condition for u} satisfies 
	\begin{equation}\label{boundedness of u}
		0 < \varphi < \mfm,
	\end{equation}
	then $\varphi$ depends on $x_2$ only, namely there exists a function $\bar{\varphi}:[0,1]\to\RR$ such that
	\begin{equation}\label{independence of x1}
		\varphi(x_1, x_2) = \bar{\varphi}(x_2) \quad \text{for} \,\, (x_1, x_2)\in \Tilde{\Omega}.
	\end{equation}
Furthermore, $\varphi$ is strictly increasing.
\end{proposition}

Before presenting the proof of Proposition~\ref{main_theorem_2}, we state  the following lemma without proof, which can be regarded as a special case of \cite[Theorem~1.5(iii)]{Berestycki_Nirenberg_1}. 
\begin{lemma}\label{Monotonicity of ODE}
	Assume that $\mff$ is Lipschitz continuous function, $\eta_0$ and $\eta_1$ are two constants. 	Let $\eta\in C^2(0,1) \cap C[0, 1]$ be a solution to the boundary value problem
	\begin{equation*}
		\begin{cases}
			\eta'' = \mff(\eta) \quad\text{in}\quad(0, 1),\\
			\eta(0) = \eta_0,\quad\eta(1) = \eta_1,
		\end{cases}
	\end{equation*}
	such that
	\begin{equation*}
		\eta_0 < \eta (x_2) < \eta_1\quad\text{for all}\quad x_2\in(0, 1).
	\end{equation*}
Then $\eta$ is strictly increasing and unique.
\end{lemma}

\begin{proof}[Proof of Proposition~\ref{main_theorem_2}]
 For each fixed $x_1$, denote $\varphi^{x_1}(x_2):=\varphi(x_1,x_2)$. Then the problem \eqref{function relation of hydrostatic euler}--\eqref{boundedness of u} can be written as the boundary value problem 
\begin{equation}\label{equation: semi-linear ODE for fixed x_1}
	\left\{
	\begin{aligned}
		&(\varphi^{x_1})'' = \mff(\varphi^{x_1}) ,\\
		&\varphi^{x_1}(0) = 0,\quad \varphi^{x_1}(1) = \mfm,
	\end{aligned}
	\right.
\end{equation}
under the constraint
\begin{equation}\label{boundedness of u_{x_1}}
	0 < \varphi^{x_1} < \mfm.
\end{equation}
As a solution of boundary value problem \eqref{equation: semi-linear ODE for fixed x_1}, the solution $\varphi^{x_1}$ is of class $C^2(0,1)\cap C[0,1]$. Hence  the uniqueness of solutions to the problem \eqref{equation: semi-linear ODE for fixed x_1}--\eqref{boundedness of u_{x_1}} follows from Lemma~\ref{Monotonicity of ODE}. Therefore, we  obtain the same solution ${\varphi}^{x_1}=\bar{\varphi}(x_2)$ for all different $x_1$'s, which implies that the solution to the problem \eqref{function relation of hydrostatic euler}--\eqref{boundedness of u}  depends only on the $x_2$-coordinate, namely \eqref{independence of x1}. It also follows from Lemma~\ref{Monotonicity of ODE} that ${\varphi}^{x_1}$ is strictly increasing, so is $\bar{\varphi}$ in \eqref{independence of x1}. This finishes the proof of the proposition.
\end{proof}

Now we are ready to give the proof of Theorem~\ref{main result}.
\begin{proof}[Proof of Theorem~\ref{main result}]\label{Proof of Theorem 1-1} As mentioned in Remark \ref{remark2.1}, without loss of generality, we assume that $v_1$ satisfies \eqref{v1positive}. Hence Lemmas \ref{endpoint_lemma} and \ref{lemma about the boundedness of u} hold.
	We recall that $\sigma : [0, \tau] \to \overline{\Omega}$ is the solution to the gradient flow \eqref{sigmaprime}, subject to the boundary conditions
	\[
\sigma(0) = (0,0)\quad \text{and}\quad \sigma(\tau) = (\xi, 1)\,\,\text{for some real number } \xi.
\] It follows from the standard theory of ordinary differential equations that $\sigma(t)\in C^3([0, \tau])$  by \eqref{sigmaprime} and the regularity of the underlying velocity field. 
	
	Define the function $\theta(t)$ in the closed interval $[0, \tau]$ as 
	\begin{equation}\label{thetadef}
	\theta(t)=\varphi(\sigma(t)).
	\end{equation}
Note that $\theta(t)\in C^1([0,\tau])$ satisfies
	\begin{equation}\label{thetaprime>0}
	    \theta'(t) = \left|\gra\varphi(\sigma(t))\right|^2 \geq (\underset{x\in\Omega}{\inf}|v|)^2 > 0 \quad \text{for } t\in [0, \tau].
	\end{equation}
	Hence it follows from the inverse function theorem that $\theta : [0, \tau] \to [0, \mfm]$, which is defined by \eqref{thetadef}, is a $C^1$-diffeomorphism from $[0, \tau] \to [0, \mfm]$.
	
	Define $\mff:[0,\mfm]\to\RR$ by
	\begin{equation*}
		\mff(\zeta):= \pd{\varphi}{x_2 x_2}(\sigma(\theta^{-1}(\zeta))).
	\end{equation*}
	Then
	\begin{equation*}
		\mff(\theta(t)) = \pd{\varphi}{x_2 x_2}(\sigma(t)),\quad \text{for any }t \in [0, \tau].
	\end{equation*}
	It is worth noting that for any $\mff\in C^1([0,\mfm])$,  it admits a $C^1$ extension $\tilde{\mff}$ over $\RR$ defined as follows
	\begin{equation*}
		\tilde{\mff}(\zeta) =
		\begin{cases}
			\arctan(\mff'(0)\zeta) + f(0)\quad&\text{for}\,\zeta\in(-\infty, 0),\\
			\mff(\zeta)\quad&\text{for}\,\zeta\in[0,\mfm],\\
			\arctan(\mff'(\mfm)(\zeta-\mfm)) + f(\mfm)\quad&\text{for}\,\zeta\in(\mfm, \infty).
		\end{cases}
	\end{equation*}
	
	Now, it remains to show that the stream function $\varphi$ is a classical solution to the degenerate PDE \eqref{function relation of hydrostatic euler}. It follows from \eqref{equation : conservation of vorticity} that the vorticity $\omega$ must be constant along any streamline of the underlying velocity field $v$. On the other hand, Lemma~\ref{endpoint_lemma} and \eqref{thetaprime>0} assert the unique existence of $t_x \in [0, \tau]$ such that $\sigma(t_x) \in \Gamma_x$ and
	\begin{equation}\label{reason for last identity}
		\theta(t_x) = \varphi(\sigma(t_x)) = \varphi(x).
	\end{equation}
	%Now, let $x$ be an arbitrary point in $\overline{\Omega}$. It is proved in Lemma~2.6 of \cite{Hamel_Nadirashvili_1} that the regular streamline $\Gamma_x$ (which is the streamline of the velocity field passing through $x$) intersects with $\Sigma = \sigma([0, \tau])$, and there exists a 
	The situation is depicted in Figure~\ref{fig:illustraction_of_sigma(t_x)}. It is worth noting that for any $y \in\Gamma_x$, $\varphi(y)=\varphi(x)$.
    \begin{figure}
	\centering
	\includegraphics[scale=0.15]{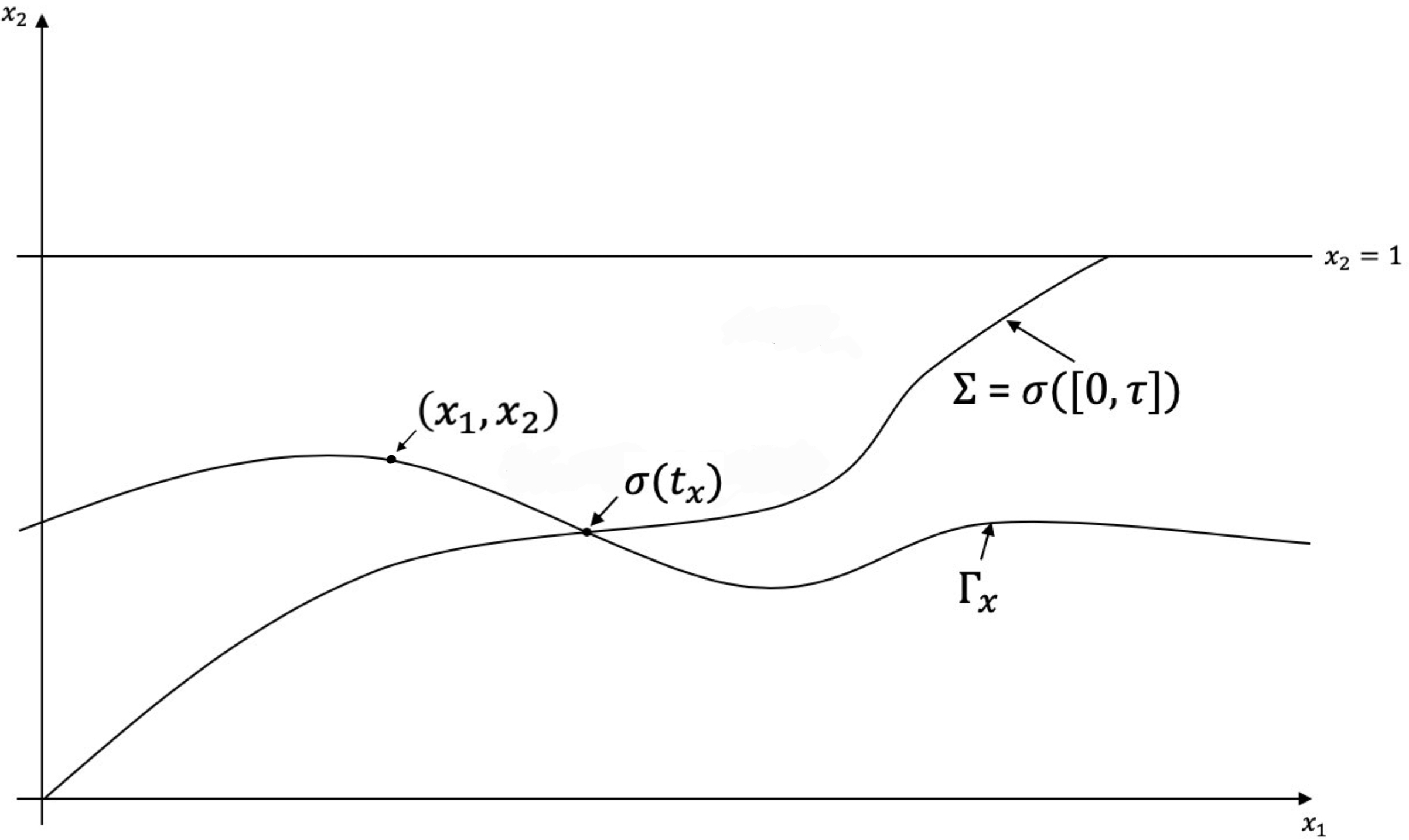}
	\caption{Illustration of $\sigma(t_x)$}
	\label{fig:illustraction_of_sigma(t_x)}
    \end{figure}
	
	Finally, it follows from the relation of the stream function $\varphi$ and the vorticity, namely
	\begin{equation*}
		\pd{\varphi}{x_2 x_2} = \pd{v_1}{x_2} = \omega,
	\end{equation*}
	that $\omega$ is constant on the streamline $\Gamma_x$ which contains both $x$ and $\sigma(t_x)$. Therefore, it follows from the definitions of $\theta$ and $\mff$, and \eqref{reason for last identity} that
	\begin{equation*}
		\begin{aligned}
			\pd{\varphi}{x_2 x_2} (x) &= \omega(x) = \omega(\sigma(t_x))\\
			&=\pd{\varphi}{x_2 x_2} (\sigma(t_x)) =\mff(\theta(t_x)) =\mff(\varphi(x)),
		\end{aligned}
	\end{equation*}
	where the second equality holds since $x$ and $\sigma(t_x)$ are on the same $\Gamma_x$.
It follows from previous calculations and Lemma~\ref{lemma about the boundedness of u} that the stream function $\varphi$ is a classical solution to the boundary value problem \eqref{function relation of hydrostatic euler}--\eqref{boundedness of u}, for some $C^1$ function $\mff:\RR\to\RR$. With the aid of Proposition~\ref{main_theorem_2}, we can conclude that $\varphi$ is independent of the variable $x_1$, namely there exists an $\bar{\varphi}\in C^3([0,1])$ such that
	\begin{equation*}
		\varphi(x) = \bar{\varphi}(x_2)	\quad \text{in }\overline{\Omega}.
	\end{equation*}
Therefore, by the definition of stream function, we have
	\begin{equation*}
		v(x) = (\bar{\varphi}'(x_2), 0)\quad\text{in}\ \overline{\Omega}.
	\end{equation*}
	Moreover, one has $\bar{\varphi}'(x_2) > 0$ for all $x_2 \in (0,1)$ by Proposition~\ref{main_theorem_2}. Hence the proof of Theorem~\ref{main result} is completed.
\end{proof}
\begin{remark}
It should be noted  that under the assumptions of Theorem \ref{main result}, $v_1(x) = \bar{\varphi}'(x_2)$ must be strictly positive (or strictly negative) in $\overline{\Omega}$, even on the boundary $\partial \Omega$. This is because $v$ is assumed to be continuous and $|v|$ has a strictly positive lower bound in $\bar\Omega$.
\end{remark}

\section{Flows in a general infinitely long nozzle}\label{section: Existence of steady state hydrostatic flows in infinite nozzles}
In this section, we study the existence, uniqueness, some fine properties of solutions to the boundary value problem \eqref{equation: elliptic PDE in perturbed domain without flattening} for the stream function. As we discussed in Section~\ref{Preliminary}, the steady hydrostatic Euler equations can be reformulated as a degenerate elliptic equation, so the key issue is to study the existence, uniqueness, and asymptotic behavior of solutions for this degenerate equation. 

First, the spatial domain $\Omega_p$ can be flattened by using a change of variables, which transforms the original boundary value problem \eqref{equation: elliptic PDE in perturbed domain without flattening} to a boundary value problem in $\Omega_0:=\RR\times(0,1)$ (cf. \eqref{transformed boundary value problem}). Second, we regard the new horizontal variable (i.e., $y_1$) as a parameter and observe that for any fixed $y_1\in\RR$, we only need to solve a boundary value problem for second-order semilinear ODE, to which the existence of $C^2$ (in the new vertical variable $y_2$) solutions can be established. Furthermore, the condition \eqref{assumption: sign condition of f on the boundary} helps to show  that the solutions to the ODE problem satisfy the boundedness condition $0<\phi<\mfc$. Hence these solutions  are unique and monotone in the $y_2$-direction because of Lemma~\ref{Monotonicity of ODE}. One of the key issues is to prove the regularity of solutions with respect to $x_1$, which is obtained by studying the associated equation after the Euler-Lagrange transformation.

We introduce the following change of variables:
\begin{equation}\label{change of variables}
	y_1 := x_1,\quad y_2 := \frac{x_2-s_0(x_1)}{s_1(x_1)-s_0(x_1)}.
\end{equation}
Then $\Omega_p$ can be mapped to $\Omega_0$. For any $y\in \Omega_0$, define
\begin{equation}\label{definition: definition of phi}
	\phi(y_1, y_2) := \varphi(y_1, (1-y_2)s_0(y_1) + y_2s_1(y_1))=\varphi(x_1, x_2).
\end{equation}
Then the boundary value problem  \eqref{equation: elliptic PDE in perturbed domain without flattening} becomes
\begin{equation}\label{transformed boundary value problem without extension}
	\begin{cases}
		\pd{\phi}{y_2 y_2} = s^2(y_1)f(\phi)\quad\text{in}\ \Omega_0,\\
		\phi(y_1, 0) = 0,\quad\phi(y_1, 1) = \mfc > 0,
	\end{cases}
\end{equation}
where the function $s$ and $f$ are defined in \eqref{defs} and \eqref{deff} respectively, and the constant $\mfc$ is defined in \eqref{mass flux constant}. 

Note that we do not know $\phi\in [0,\mfc]$ a priorily. The first difficulty for solving the problem \eqref{transformed boundary value problem without extension} is that the function $f$ is only well defined in $[0,\mfc]$. To cope with this issue, our strategy is to extend the function $f$ appropriately, and then show that the extended boundary value problem has a unique solution with values in $[0,\mfc]$.  In particular, we extend the function $f$ to a bounded Lipschitz function $\hat{f}:\RR\to\RR$ as follows,
\begin{equation}\label{extension: Lipschitz extension of f to F}
	\hat{f}(\phi) =
	\begin{cases}
		\displaystyle
		f(0)\quad&\text{if}\ -\infty<\phi\leq0,\\
		f(\phi) \quad&\text{if}\ 0<\phi< \mfc,\\
		\displaystyle
		f(\mfc) \quad&\text{if}\ \mfc\leq\phi<\infty.\\
	\end{cases}
\end{equation}
With the aid of this extension, the boundary value problem \eqref{transformed boundary value problem without extension} becomes
\begin{equation}\label{transformed boundary value problem}
	\begin{cases}
		\pd{\phi}{y_2 y_2} = s^2(y_1)\hat{f}(\phi)\quad\text{in}\ \Omega_0,\\
		\phi(y_1, 0) = 0,\quad\phi(y_1, 1) = \mfc > 0.
	\end{cases}
\end{equation}
One of the key problem is to prove the existence of solutions for \eqref{transformed boundary value problem} satisfying $0\leq \phi \leq \mfc$.

In the following, we prove the existence and uniqueness of solutions to the extended boundary value problem \eqref{transformed boundary value problem} for degenerate elliptic equation via the fixed-point approach.

First, let us recall the following Schauder fixed point theorem.
\begin{proposition}\label{Schauder's fixed point theorem}
	(\cite[Theorem~11.1]{Gilbarg_Trudinger})
	Let $M$ be a compact convex set in a Banach space $X$ and let $T$ be a continuous mapping from $M$ into itself. Then $T$ has a fixed point in $M$. 
\end{proposition}
To apply the Schauder fixed point theorem, we first define the mapping $T$. Note that for any fixed $y_1\in\RR$, the boundary value problem \eqref{transformed boundary value problem} can be rewritten as the integral equation
\begin{equation}\label{integral formula}
	\rho(y_2) = \int_0^1 s^2(y_1)G(\xi, y_2)\hat{f}(\rho(\xi))\,d\xi + \mfc y_2,
\end{equation}
where $G(\xi, y_2)$ is defined by
\begin{equation}\label{explicit form of Green's kernel}
	G(\xi, y_2) := 
	\begin{cases}
		\xi(y_2-1),\quad\text{if}\ 0\leq \xi \leq y_2,\\
		y_2(\xi-1),\quad\text{if}\ y_2\leq \xi \leq 1.
	\end{cases}
\end{equation}

Let $X := C^2([0,1])$, and define 
\begin{equation*}
	\begin{aligned}
		M := \bigg\{\rho\in C^2[0,1]: & \rho(0) = 0, \rho(1) = \mfc > 0, 
		[\rho'']_{0,1; [0,1]} \leq C(\hat{f}, \mfc)\\
		&\|\rho\|_{C^2[0,1]}\leq\frac{13s^2(y_1)}{8}\underset{\RR}{\sup}|\hat{f}| + 2\mfc\bigg\},
	\end{aligned}
\end{equation*}
where the Lipschitz seminorm $[g]_{0,1; I}$ for a function \(g\) on the interval $I$ and the constant $C(\hat{f}, \mfc)$ are defined by 
\[
[g]_{0,1; I}=\sup_{y_2, \tilde{y}_2\in I}\frac{|g(y_2)-g(\tilde{y}_2)|}{|y_2-\tilde{y}_2|}
\]
and 
\[
C(\hat{f},\mfc) := s^2(y_1)[\hat{f}]_{0,1; \RR} \left(\frac{13s^2(y_1)}{8}\|\hat{f}\|_{C(\RR)} + 2\mfc\right).
\]
 For any $\rho\in X$, define
\begin{equation}\label{definition of T} 
	(T(\rho))(y_2) := \int_0^1 s^2(y_1)G(\xi, y_2)\hat{f}(\rho(\xi))\,d\xi + \mfc y_2.
\end{equation}
It is easy to check that the set $M$ is a convex set in $X$. In addition, the requirements on the Lipschitz continuity of $\rho''$ guarantees that $M$ is  a compact subset of $C^2[0,1]$ by the Arzel\`{a}-Ascoli theorem. Now we are ready to prove the following proposition.
\begin{proposition}[Existence of solutions for fixed $y_1$]\label{existence of solution}
	Let $\hat{f}:\mathbb{R}\to\mathbb{R}$ be a bounded Lipschitz function. Then the boundary value problem
	\begin{equation}\label{two point boundary value problem for fixed x_1}
		\begin{cases}
			\rho'' = s^2(y_1)\hat{f}(\rho)\quad\text{in}\ [0,1],\\
			\rho(0) = 0,\quad\rho(1) = \mfc > 0
		\end{cases}
	\end{equation}
	has a solution $\rho\in M$.
\end{proposition}

\begin{proof}
The key issue is to show that the map $T$ defined in \eqref{definition of T} is a continuous map from $M$ into itself.
Indeed, it follows  from \eqref{explicit form of Green's kernel} and \eqref{definition of T} that 
	\[
	T(\rho)(0) = 0,\quad \quad T(\rho)(1) = \mfc,
	\]
	and 
\[	\|T(\rho)\|_{C[0,1]} \leq \frac{s^2(y_1)}{8}\|\hat{f}\|_{C(\RR)} + \mfc.
\]
	Furthermore, direct computations yield
	\begin{equation*}
		(T(\rho))'(y_2) = \ii{0}{y_2}\xi s^2(y_1)\hat{f}(\rho(\xi))\,d\xi +\ii{y_2}{1}(\xi-1)s^2(y_1)\hat{f}(\rho(\xi))\,d\xi + \mfc
	\end{equation*}
	and
	\begin{equation}\label{2ndderT}
		(T(\rho))''(y_2) = s^2(y_1)\hat{f}(\rho(y_2)).
	\end{equation}
	Hence one has
	\[
		\|(T(\rho))'\|_{C[0,1]} \leq \frac{s^2(y_1)}{2}\|\hat{f}\|_{C(\RR)} + \mfc
	\]
	and 
	\begin{equation*}
		\begin{aligned}		
			\|(T(\rho))''\|_{C[0,1]} &\leq s^2(y_1)\|\hat{f}\|_{C(\RR)}.
		\end{aligned}
	\end{equation*}
	Consequently, one has
	\begin{equation*}
		||T(\rho)||_{C^2[0,1]} \leq \frac{13s^2(y_1)}{8}||\hat{f}||_{C(\RR)} + 2\mfc.
	\end{equation*}
Furthermore, it follows from \eqref{2ndderT} that  for any $y_2$, $\Tilde{y}_2 \in [0,1]$, one has
	\begin{equation}\label{Tphi''Lip}
		\begin{aligned}
			&\quad |T(\rho)''(\Tilde{y}_2) - T(\rho)''(y_2)| = s^2(y_1)|\hat{f}(\rho(\Tilde{y}_2)) - \hat{f}(\rho(y_2))|\\
			&\leq s^2(y_1)[\hat{f}]_{0,1; \RR}|\rho(\Tilde{y}_2) - \rho(y_2)| 
			\leq s^2(y_1)[\hat{f}]_{0,1; \RR} [\rho]_{0,1; [0,1]} |\Tilde{y}_2-y_2|\\
			&\leq s^2(y_1)[\hat{f}]_{0,1; \RR} \left(\frac{13s^2(y_1)}{8}\|\hat{f}\|_{C(\RR)} + 2\mfc\right)|\Tilde{y}_2 - y_2|.
		\end{aligned}
	\end{equation}
	In addition, the mapping $T$ is indeed a Lipschitz mapping, namely
	\[\|T(\rho)-T(\varrho)\|_{C^2[0,1]} \leq \frac{13s^2(y_1)[\hat{f}]_{0,1; \RR}}{8}\|\rho-\varrho\|_{C^2[0,1]}.\]
Therefore, $T$ is a continuous mapping from $M$ into itself. It follows from the Schauder fixed point theorem (cf. Proposition~\ref{Schauder's fixed point theorem}) that there exists a fixed point for \eqref{definition of T}. Hence there exists a solution $\rho \in M$ to the two-point boundary value problem \eqref{two point boundary value problem for fixed x_1}. 
\end{proof}

Next, we prove that the solution is unique and has positive horizontal velocity, namely $v_1 = \frac{1}{s_1-s_0}\pd{\phi}{y_2}>0$.

\begin{proposition}\label{Theorem: Monotonicity result in x_2 by sign condition}
Let $\hat{f}$ be of form \eqref{extension: Lipschitz extension of f to F} and satisfy condition \eqref{assumption: sign condition of f on the boundary}.	Assume that there exists a classical solution to the boundary value problem \eqref{transformed boundary value problem}. Then $\phi$ is unique, and satisfies
\[
 0 \leq \phi \leq \mfc\quad  \text{and}\quad  \pd{\phi}{y_2} > 0\,\, \text{in}\,\, \RR\times[0,1].
 \]
\end{proposition}
\begin{proof} We divide the proof into three steps.
	
	{\it Step 1. $L^\infty$-estimate.}
	We first show that for any fixed $y_1 \in \RR$, the solution $\ph^{y_1}(y_2):=\phi(y_1, y_2)$ of the boundary value problem \eqref{transformed boundary value problem} actually satisfies the boundedness condition
	\begin{equation}\label{boundedness condition for phi on fixed x1}
		0<\ph^{y_1}<\mfc 	\quad \text{in } (0,1).
	\end{equation}
If  \eqref{boundedness condition for phi on fixed x1} is not true, then one of the following two cases must happen.
	\begin{itemize}
		\item [(i)] the function $\phi^{y_1}$ attains maximum at $\Tilde{y}_2\in (0,1)$ with $\phi^{y_1}(\Tilde{y}_2) = C \geq \mfc$;
		\item [(ii)] the function $\phi^{y_1}$ attains minimum at $\Tilde{y}_2\in (0,1)$ with $\phi^{y_1}(\Tilde{y}_2) = \mu \leq  0$;
	\end{itemize}
	We need to rule out the both cases.
	
	Seeking for a contradiction, we suppose that Case (i) happens, i.e., there exist a point $\Tilde{y}_2\in(0,1)$ and $C\geq \mfc$ such that
	\[\ph^{y_1}(\Tilde{y}_2) =\max_{[0,1]}\ph^{y_1}({y}_2) = C.
	\]
	One has
	\[
	\frac{d}{dy_2}\phi^{y_1}(\tilde{y}_2)=0\quad\text{and}\quad  	\frac{d^2}{d y_2^2}\phi^{y_1}(\tilde{y}_2)\leq 0.
	\]
It follows from ${\eqref{transformed boundary value problem}}_1$ that one has 
\[
\hat{f}(C) =\frac{1}{s^2(y_1)}\frac{d^2}{dy_2^2}\phi^{y_1}(\tilde{y}_2)\leq 0.
\]
 This, together with \eqref{assumption: sign condition of f on the boundary} and \eqref{extension: Lipschitz extension of f to F}, implies  $\hat{f}(C)=f(\mfc)=0$.

	On the other hand, $\phi^{y_1}(y_2)$ satisfies the following initial value problem
	\begin{equation*}
		\begin{cases}
			(\phi^{y_1})'' = s^2(y_1)\hat{f}(\phi^{y_1})\\
			\phi^{y_1}(\Tilde{y}_2) = C,\quad (\phi^{y_1})'(\Tilde{y}_2) = 0,
		\end{cases}
	\end{equation*}
	which has a unique solution $\ph^{y_1} \equiv C$ in $[0,1]$ since $\hat{f}$ is a bounded Lipschitz function. This implies that $\ph^{y_1}(0) = C>0$, which contradicts the given boundary condition $\ph^{y_1}(0) = 0$. Hence, case (i) is impossible.

Similarly, we can show that Case (ii) cannot happen, either. 
 Therefore, $\ph^{y_1}$ can never reaches its maximum and minimum in the interior of $(0,1)$, so \eqref{boundedness condition for phi on fixed x1} must hold.
	
	In addition, it follows from Lemma~\ref{Monotonicity of ODE} that $\phi^{y_1}$ is unique, and so is $\phi$.  We can also verify that $\phi$ is indeed the unique solution to \eqref{transformed boundary value problem without extension}, provided that $f(0)\leq 0 \leq f(\mfc)$.
	
	{\it Step 2. Away from stagnation inside the domain.}
	It follows from Lemma~\ref{Monotonicity of ODE} that $\ph^{y_1}$ is strictly increasing in $y_2$. For any $h\in(0,\frac{1}{2}]$, define 
	\begin{equation*}
		W(y_2) := \ph^{y_1}(2h-y_2) - \ph^{y_1}(y_2)\quad \text{for any } y_2\in [0,h].
	\end{equation*}
Clearly, $W(h) = 0$. Furthermore, since $\ph^{y_1}$ is strictly increasing, one also has $W > 0$ in $[0,h)$. Indeed, $W$ also satisfies
	\[W'' + s^2(y_1)g(y_2) W = 0,\]
	where
	\[g(y_2) := \frac{f(\ph^{y_1}(2h-y_2)) - f(\ph^{y_1}(y_2))}{\ph^{y_1}(2h-y_2) - \ph^{y_1}(y_2)}.\]
	Then by the Hopf's lemma, we conclude that $0>W'(h) = -2(\ph^{y_1})'(h)$, or equivalently
	\[(\ph^{y_1})'(h) > 0\]
	for all $h\in(0,\frac{1}{2}]$. To prove that $(\ph^{y_1})' > 0$ in $[\frac{1}{2},1)$, we only need to apply the above argument to the function $\mfc - \ph^{y_1}(1-y_2)$ instead. Since $y_1$ is arbitrary, the above assertion also implies $(\ph^{y_1})' > 0$ in $\mathbb{R}\times(0,1)$. 
	
	{\it Step 3. Away from the stagnation on the boundary.}   Now we are ready to prove the strict inequality holds up to the end points $y_2 = 0$ or $1$. Let us consider the case $y_2 = 1$ and recall that $\phi^{y_1}(1) = \mfc$. If $f(\mfc) > 0$, then there exists a small neighborhood of $y_2 = 1$ in $(0,1)$ such that
	\[(\phi^{y_1})'' = f(\phi^{y_1}) >0.\]
	In addition, we know that $\phi^{y_1} < \mfc$ in this neighbourhood. Hence, we can conclude that $(\phi^{y_1})'(1) > 0$ by the Hopf lemma. On the other hand, if $f(\mfc) = 0$, then $\phi^{y_1}$ satisfies
	\[(\phi^{y_1}-\mfc)''+a(y_1)(\phi^{y_1}-\mfc) = 0,\]
	where
	\[a(y_1) := \frac{f(\phi^{y_1})-f(\mfc)}{\phi^{y_1}-\mfc}\]
	is a bounded function. 
	Then it follows from the Hopf lemma that
	\[(\phi^{y_1})'(1) > 0.\]
	Similarly, one can also show that $(\phi^{y_1})' > 0$ at $y_2 = 0$.
\end{proof}

It follows from Proposition~\ref{Theorem: Monotonicity result in x_2 by sign condition} that we have the following corollary.
\begin{corollary}
	 The unique solution to \eqref{transformed boundary value problem} is also the unique solution to \eqref{transformed boundary value problem without extension}, provided that $f(0) \leq 0 \leq f(\mfc)$.
\end{corollary}

In addition, since $\varphi^{-}$ defined in \eqref{equation: relation of stream function and horizontal velocity in upstream} is differentiable and $(\varphi^{-})' = v_1^- > 0$,
the function $\kappa$ defined in \eqref{definition: kappa} is differentiable and 
\[\kappa'(\varphi) = \frac{1}{(v_1^-)(\kappa(\varphi))}.\]
The straightforward differentiation yields
\begin{equation}\label{equality: explicit form of f prime}
	f'(\cdot) = (v_1^-)''(\kappa(\cdot))\kappa'(\cdot) = \frac{(v_1^-)''(\kappa(\cdot))}{(v_1^-)(\kappa(\cdot))}.
\end{equation}

After showing the $y_2$-monotonicity of $\phi$, we prove the non-degeneracy of $\pd{\phi}{y_2}$ to the boundary value problem \eqref{transformed boundary value problem}. In contrast to previous results in this section, all the statements and proofs below, especially for the far field behaviors, rely on Assumption~\ref{structural assumption on Omega p}.
\begin{proposition}\label{proposition: positive lower bound of phi prime at far field}
	For any fixed $y_1\in\RR$, let $\phi^{y_1}(y_2)$ be the solution obtained in Proposition \ref{Theorem: Monotonicity result in x_2 by sign condition}. Then there exists a constant $\gamma>0$ such that
	\begin{equation*}
		\liminf_{y_1\to\pm\infty}(\phi^{y_1})' \geq \gamma >0.
	\end{equation*}
\end{proposition}
\begin{proof}
	It follows from Proposition~\ref{Theorem: Monotonicity result in x_2 by sign condition} that for any $y_1\in \RR$,
	\[(\phi^{y_1})'(y_2) > 0\quad \text{for all}\,\, y_2\in [0,1].\]
Since the analysis for flows in the upstream and downstream is the same,	without lost of generality, it suffices to show that there exists a constant $\gamma_->0$ such that
	\begin{equation}\label{4.14.5}
		\liminf_{y_1\to-\infty}(\phi^{y_1})' \geq \gamma_-.
	\end{equation}
Seeking for a contradiction, we assume that \eqref{4.14.5} is not true. Then 
	\begin{equation}\label{Assumption for lim inf vanishing}
		\lim_{y_1\to-\infty}\inff{[0,1]}(\phi^{y_1})'(y_2) = 0.
	\end{equation}
	Hence there exists a sequence $\{y_1^m\}_{m=1}^\infty$, which tends to $-\infty$ as $m\to\infty$, such that
	\begin{equation}\label{limphin0}
	\inff{[0,1]}(\phi^{y_1^m})'(y_2) \to 0\quad\text{as}\ m\to\infty.
\end{equation}
	By Proposition~\ref{existence of solution} and the Arzel\`{a}-Ascoli theorem,  there exists a subsequence (still labelled by $\{\phi^{y_1^m}\}_{m=1}^\infty$) which converges, in $C^2[0,1]$, to some function $\phi^{-\infty}$. We claim that  $\phi^{-\infty}$ satisfies
	\begin{equation}\label{problem: ODE for the uniform limit of phi}
		\begin{cases}
			(\phi^{-\infty})'' = f(\phi^{-\infty})\quad\text{in}\ (0,1)\\
			\phi^{-\infty}(0) = 0,\quad\phi^{-\infty}(1) = \mfc > 0.
		\end{cases}
	\end{equation}

	To prove the claim, it is worth noting that $\phi^{-\infty}(y_2)$ can be rewritten as
	\begin{equation*}
		\begin{aligned}
			\phi^{-\infty}(y_2) = & \lim_{m\to\infty}\ph^{y_1^m}(y_2)\\ 
			= &\lim_{m\to\infty}\int_{0}^{1}\frac{1}{2}(|y_2 - \zeta| - y_2 - \zeta + 2y_2 \zeta)s^2(y_1^m)f(\ph^{y_1^m}(\zeta))\,d\zeta + \mfc y_2.
		\end{aligned}
	\end{equation*}
	Furthermore, it follows from Assumption~\ref{structural assumption on Omega p} and the Lipschitz continuity of $f$ that 
	\[I_{y_1^m}(\zeta)=\frac{1}{2}(|y_2 - \zeta| - y_2 - \zeta + 2y_2 \zeta)s^2(y_1^m)f(\ph^{y_1^m}(\zeta))\]
	converges uniformly to
	\[I_{-\infty}(\zeta):=\frac{1}{2}(|y_2 - \zeta| - y_2 - \zeta + 2y_2 \zeta)f(\ph^{-\infty}(\zeta)),\]
	as $m\to\infty$.
	Moreover, for any $m\in \mathbb{N}$, it follows from Proposition~\ref{Theorem: Monotonicity result in x_2 by sign condition} that $0\leq\phi^{y_1^m}\leq\mfc$, so
	\begin{equation*}
		\begin{aligned}
			|I_{y_1^m}(\zeta)|\leq\frac{1}{4} (\ssu{\RR}s^2)(\ssu{[0,\mfc]}|f|) < \infty.
		\end{aligned}
	\end{equation*} 
	Hence, by the bounded convergence theorem, one has
	\[\phi^{-\infty}(y_2) = \int_{0}^{1}\frac{1}{2}(|y_2 - \zeta| - y_2 - \zeta + 2y_2 \zeta)f(\ph^{-\infty}(\zeta))\,d\zeta + \mfc y_2.\]
	This implies that $\phi^{-\infty}$ is a classical solution to \eqref{problem: ODE for the uniform limit of phi} as desired.
	Now, it follows from Proposition~\ref{Theorem: Monotonicity result in x_2 by sign condition} again that
	\[(\phi^{-\infty})' > 0\quad\text{in}\ [0,1].\]
	Hence there exists a constant  $\gamma_->0$ such that  
	\begin{equation*}
		\lim_{y_1\to-\infty}(\phi^{y_1})' \geq \gamma_- >0,
	\end{equation*}
	which contradicts \eqref{limphin0}.

	By the same argument, one can also prove that there exists a constant $\gamma_+>0$ such that
	\begin{equation*}
		\lim_{y_1\to\infty}(\phi^{y_1})' \geq \gamma_+ >0.
	\end{equation*}
	Define $\gamma:=\min\{\gamma_-, \gamma_+\}>0$. This finishes the proof of the proposition.
\end{proof}

\begin{corollary}\label{corollary: uniform positive lower bound for phi x2}
	Let $\phi$ be the solution to \eqref{transformed boundary value problem without extension}. Assume that $\pd{\phi}{y_2}$ is continuous in $\Om$, then there exists a constant $\bar\gamma>0$ such that
	\[\pd{\phi}{y_2} \geq \bar\gamma \quad\text{in}\ \Om.\]
\end{corollary}

In order to show that the velocity $v$ induced by the stream function is indeed a classical solution to the hydrostatic Euler equations, we need to study the $y_1$-regularity of the unique solution $\phi$ to \eqref{transformed boundary value problem without extension}, provided that $f(0)\leq 0\leq f(\mfc)$.

We first introduce the Euler-Lagrange transformation from $(y_1, y_2)$ to $(z_1, z_2)$ as follows: 
\begin{equation}\label{definition: E-L transformation}
	(z_1, z_2) := \mathcal{L}(y_1, y_2):= (y_1, \phi(y_1, y_2)).
\end{equation}
It is worth noting that this transformation has an inverse since Proposition~\ref{proposition: positive lower bound of phi prime at far field} implies that the unique solution of \eqref{transformed boundary value problem without extension} satisfies
$
\pd{\phi}{y_2} > 0.
$
Under this change of variables, one has
\begin{equation}\label{identity: function compposite with inverse}
	\phi(z_1, \Phi(z_1,z_2)) = z_2,
\end{equation}
where $\Phi(z_1, z_2)$ is the inverse function of the mapping $z_2\mapsto\phi (z_1,z_2)$ when $z_1$ is fixed. Differentiating \eqref{identity: function compposite with inverse} with respect to $z_2$ yields 
\begin{equation}\label{identity: first derivative of phi and phi inverse}
	\pd{\phi(z_1, \Phi(z_1,z_2))}{y_2}\pd{\Phi(z_1,z_2)}{z_2} = 1,
\end{equation}
which is equivalent to
\begin{equation}\label{dery2phi}
	\pd{\phi(z_1, \Phi(z_1,z_2))}{y_2} = \frac{1}{\pd{\Phi(z_1,z_2)}{z_2}}.
\end{equation}
Furthermore, differentiating \eqref{identity: first derivative of phi and phi inverse} with respect to $z_2$ yields
\begin{equation}\label{identity: second derivative of phi and phi inverse}
	\pd{\phi(z_1, \Phi(z_1,z_2))}{y_2 y_2}(\pd{\Phi(z_1,z_2)}{z_2})^2 + \pd{\phi(z_1, \Phi(z_1,z_2))}{y_2} \pd{\Phi(z_1,z_2)}{z_2z_2} = 0.
\end{equation}
Hence it follows from \eqref{identity: first derivative of phi and phi inverse} and \eqref{identity: second derivative of phi and phi inverse} that
\begin{equation*}
	\pd{\phi(z_1, \Phi(z_1,z_2))}{y_2 y_2} = -\frac{\pd{\Phi(z_1,z_2)}{z_2z_2}}{(\pd{\Phi(z_1,z_2)}{z_2})^3}.
\end{equation*}
Therefore, problem~\eqref{transformed boundary value problem without extension} can be transformed into the following problem:
\begin{equation}\label{problem: degenerate elliptic PDE for phi inverse}
	\begin{cases}
		\pd{\Phi(z_1,z_2)}{z_2z_2} + s^2(z_1)f(z_2)(\pd{\Phi(z_1,z_2)}{z_2})^3 = 0,\\
	\Phi(z_1,0) = 0,\quad\Phi(z_1,\mfc) = 1.
	\end{cases}
\end{equation}
{We note that this boundary value problem has a unique solution due to the uniqueness of solution to problem~\eqref{transformed boundary value problem without extension}.}
We have the following proposition about the regularity of $\Phi$.
\begin{proposition}\label{theorem: regularity of phi inverse}
	Let $\Phi(z_1,z_2)$ be the classical solution of \eqref{problem: degenerate elliptic PDE for phi inverse}, then both $\pd{\Phi(z_1,z_2)}{z_1}$ and $\pd{\Phi(z_1,z_2)}{z_1z_2}$ exist. Furthermore, $\Phi(z_1,z_2)$, $\pd{\Phi(z_1,z_2)}{z_1}$, $\pd{\Phi(z_1,z_2)}{z_2}$ and $\pd{\Phi(z_1,z_2)}{z_1z_2}$ are jointly continuous with respect to both $z_1$ and $z_2$. Finally, it holds that
	\begin{equation}\label{asymptPhi}
		\pd{\Phi(z_1, z_2)}{z_1} \to 0 \quad \text{uniformly as }z_1\to\pm\infty.
	\end{equation}
\end{proposition}
\begin{proof}
	We divide the proof into two steps.
	
{\it Step 1. Regularity of $\Phi$.}  Define $\psi:= \pd{\Phi(z_1, z_2)}{z_2}$. Then $\psi$ satisfies
	\begin{equation*}
		\pd{\psi}{z_2} + s^2(z_1)f(z_2)\psi^3 = 0.
	\end{equation*}
Hence $\psi$ can be solved explicitly as follows
	\begin{equation}\label{equation: explicit form of psi square}
		\psi(z_1,z_2)  = \left(2s^2(z_1)\int_0^{z_2}f(\Tilde{\zeta})\,d\Tilde{\zeta} + \beta(z_1)\right)^{-\frac{1}{2}},
	\end{equation}
where $\beta(z_1)$ depends only on $z_1$. It follows from Propositions~\ref{existence of solution} and \ref{proposition: positive lower bound of phi prime at far field} that 
	\[
	0<\frac{1}{\psi}=\pd{\phi}{y_2}\leq \frac{13s^2(y_1)}{8}\|f\|_{C(\RR)} + 2\mfc.
	\]
Using \eqref{dery2phi} yields
	\[2s^2(z_1)\int_0^{z_2}f(\Tilde{\zeta})\,d\Tilde{\zeta} + \beta(z_1) = (\psi(z_1,z_2))^{-2} > 0.\]
	Hence it follows from \eqref{equation: explicit form of psi square}, $\Phi(z_1,0) = 0$ and the definition of $\psi$ that
	\begin{equation*}
	\Phi(z_1,z_2) = \int_0^{z_2}\frac{1}{\sqrt{2s^2(z_1)\int_0^{\zeta} f(\Tilde{\zeta})\,d\Tilde{\zeta} + \beta(z_1)}}\,d\zeta.
	\end{equation*}
	In particular, at $z_2 = \mfc$, one has
	\begin{equation}\label{equation:G(s^2,beta)=1}
		1 = \Phi(z_1,\mfc) = \int_0^{\mfc}\frac{1}{\sqrt{2s^2(z_1)\int_0^{\zeta}f(\Tilde{\zeta})\,d\Tilde{\zeta} + \beta(z_1)}}\,d\zeta =\mathcal{G}(s^2(z_1),\beta(z_1)),
	\end{equation}
	where
	\begin{equation*}
		\mathcal{G}(\alpha_1,\alpha_2):= \int_0^{\mfc}\frac{1}{\sqrt{2\alpha_1\int_0^{\zeta}f(\Tilde{\zeta})\,d\Tilde{\zeta} + \alpha_2}}\,d\zeta.
	\end{equation*}
	Differentiating $\mathcal{G}(\alpha_1,\alpha_2)$ with respect to $\alpha_2$, and then evaluating at $(\alpha_1,\alpha_2)=(s^2(z_1),\beta(z_2))$ yield
	\begin{equation*}
		\pd{\mathcal{G}}{\alpha_2}(s^2(z_1),\beta(z_1)) = -\frac{1}{2}\left.\int_0^\mfc\frac{1}{(2\alpha_1\int_0^{\zeta}f(\Tilde{\zeta})\,d\Tilde{\zeta} + \alpha_2)^\frac{3}{2}}\,d\zeta\right|_{(\alpha_1,\alpha_2)=(s^2(z_1),\beta(z_1))} < 0.
	\end{equation*}
	Hence, applying the implicit function theorem to \eqref{equation:G(s^2,beta)=1} yields that $\beta(z_1) = \Bar{\beta}(s^2(z_1))$ for some  function $\Bar{\beta}(\cdot)\in C^1(\RR)$. Hence one has 
	\begin{equation}\label{explicit formula of phi inverse}
	\Phi(z_1,z_2) = \int_0^{z_2}\frac{1}{\sqrt{2s^2(z_1)\int_0^{\Bar{\zeta}}f(\Tilde{\zeta})\,d\Tilde{\zeta} + \Bar{\beta}(s^2(z_1))}}\,d\Bar{\zeta}.
	\end{equation}
	Since both $s$ and $\Bar{\beta}$ are continuous, we can conclude from \eqref{explicit formula of phi inverse} that $\Phi$ is continuous with respect to both the $z_1$ and $z_2$-coordinate. In addition, equation~\eqref{explicit formula of phi inverse} also implies that $\pd{\Phi}{z_1}$, $\pd{\Phi}{z_2}$, and $\pd{\Phi}{z_1z_2}$ exist and are continuous. 
	
{\it Step 2. Asymptotic behavior of $\partial_{z_1}\Phi$.} As mentioned in Remark~\ref{remark: remark on the boundedness of the height of nozzle}, there exist two constants $\underline{d}$ and $\bar{d}$ with $0<\underline{d}\leq \bar{d}<+\infty$ such that
	\[\underline{d}^2 \leq s^2(z_1) \leq \bar{d}^2.\]
It follows from equation~\eqref{identity: first derivative of phi and phi inverse} and Corollary~\ref{corollary: uniform positive lower bound for phi x2} that
	\[\ssu{\RR\times[0,\mfc]}\left|\pd{\Phi(z_1, z_2)}{z_2}\right| = \ssu{\overline{\Om_0}}\left|\frac{1}{\pd{\phi}{y_2}}\right| \leq \frac{1}{\bar\gamma} < \infty.\]
	Direct calculations yield
	\begin{equation*}
		\begin{aligned}
			&\quad|\pd{\Phi(z_1,z_2)}{z_1}|\\
			&\leq \bigg|\frac{1}{2}\int_0^{z_2}\left(2(s^2)'(z_1)\int_0^{\zeta}f(\Tilde{\zeta})\,d\Tilde{\zeta} + \Bar{\beta}'(s^2(z_1))(s^2)'(z_1))\right)(\pd{\Phi(z_1,z_2)}{z_2})^{3}\,d\zeta\bigg|\\
			&\leq\frac{\mfc}{2}\left(2\mfc~ \ssu{[0,\mfc]}|f|+\ssu{[\underline{d}^2,\bar{d}^2]}|\Bar{\beta}'|\right)\frac{|(s^2)'(z_1)|}{\bar{\gamma}^3}.
		\end{aligned}
	\end{equation*}
 Hence, by Assumption~\ref{structural assumption on Omega p}, we conclude \eqref{asymptPhi}.
	This finishes the proof of the proposition.
\end{proof}

Let $Y$ be defined as
\begin{equation}\label{definition: definition of function space Y}
	Y := \{\phi: \phi, \pd{\phi}{y_1}, \pd{\phi}{y_2}, \pd{\phi}{y_1 y_2},\pd{\phi}{y_2 y_2}\in C(\bbar{\Om_0}), \phi(y_1, 0) = 0, \phi(y_1, 1) = \mfc\}
\end{equation}
equipped with the norm
\begin{equation*}
	||\phi||_{Y} := \ssu{\bbar{\Om_0}}(|\phi| + |\pd{\phi}{y_1}| + |\pd{\phi}{y_2}| + |\pd{\phi}{y_1 y_2}| +|\pd{\phi}{y_2 y_2}|).
\end{equation*}
With the above propositions, we are ready to prove the main result for the boundary value problem \eqref{transformed boundary value problem without extension}.
\begin{proposition}\label{theorem: existence of stream function without sign condition}
	Let the function $f$ satisfy \eqref{assumption: sign condition of f on the boundary}.
	Then the boundary value problem \eqref{transformed boundary value problem without extension} has a unique solution $\phi\in Y$.
	Furthermore, there exists a constant $\hat{\gamma}>0$ such that
	\begin{equation}\label{e:dy2phigeqGamma>0}
		\pd{\phi}{y_2} \geq \hat\gamma > 0 \quad \text{in}\,\, \RR\times[0,1],
	\end{equation}
	and
	\begin{equation}\label{eq4.28.5}
		\pd{\phi}{y_1}\to0 \quad \text{uniformly as}\,\, y_1\to\pm\infty.\end{equation}
\end{proposition}
\begin{proof}
	The regularity and monotoncity in the $y_2$-direction of the solution $\phi$ to \eqref{transformed boundary value problem} follow immediately from Propositions~\ref{existence of solution} and ~\ref{Theorem: Monotonicity result in x_2 by sign condition}. It follows from Proposition~\ref{Theorem: Monotonicity result in x_2 by sign condition} again that the unique solution $\phi$ takes value only in $[0,\mfc]$. Hence, the unique solution $\phi$ of \eqref{transformed boundary value problem} also satisfies \eqref{transformed boundary value problem without extension} indeed. 
	
	We note that the inverse of the transformation \eqref{definition: E-L transformation} is
	\begin{equation}\label{inverse of L}
		(y_1, y_2) = (z_1, \Phi(z_1, z_2)),
	\end{equation}
	which is a bijective continuous function from $\RR\times(0,\mfc)$ to $\Omega_0=\mathbb{R}\times (0,1)$. Hence, by the invariance of domain, we can conclude that the transformation \eqref{definition: E-L transformation} is jointly continuous with respect to both $y_1$ and $y_2$, and so is $\phi(y_1, y_2)$. Furthermore,  $\pd{\phi}{y_2}$ is also continuous since we have relation \eqref{identity: first derivative of phi and phi inverse} and Proposition~\ref{theorem: regularity of phi inverse}.
	
	Differentiating \eqref{identity: function compposite with inverse} with respect to $z_1$, and using \eqref{inverse of L}, we obtain
	\begin{equation}\label{identity: relation of x derivative of phi and phi inverse}
		\pd{\phi(y_1, y_2)}{y_1} = \pd{\phi(z_1, \Phi(z_1,z_2))}{{y_1}} %\color{blue} 
		= -(\pd{\phi}{y_2}(z_1, \Phi(z_1, z_2)))(\pd{\Phi(z_1,z_2)}{z_1}).
	\end{equation}
	This means that $\pd{\phi}{y_1}$ exists and is continuous, provided that $\pd{\phi}{y_2}$ and $\pd{\Phi}{z_1}$ are well-defined and continuous, which has been just proved and was proved in Proposition~\ref{theorem: regularity of phi inverse}, respectively.
	
	The jointly continuity of $\pd{\phi}{y_2 y_2}$ follows immediately from the equation in 
	\eqref{transformed boundary value problem without extension}
	and the fact that $\phi$ is jointly continuous with respect to both the $y_1$ and $y_2$-coordinate.
	Furthermore, differentiating \eqref{identity: first derivative of phi and phi inverse} with respect to $z_1$, and using \eqref{dery2phi} and \eqref{inverse of L}, we have
	\begin{align*}
		\pd{\phi}{y_1 y_2}(y_1,y_2) &= \pd{\phi}{y_1 y_2}(z_1,\Phi(z_1,z_2))\\
		&= -\frac{\pd{\Phi(z_1,z_2)}{z_1z_2}}{(\pd{\Phi(z_1,z_2)}{z_2})^2} - \pd{\phi}{y_2 y_2}(z_1,\Phi(z_1,z_2))\pd{\Phi(z_1,z_2)}{z_1}.
	\end{align*}
	Hence one can conclude that $\pd{\phi}{y_1 y_2}$ exists and is continuous by Proposition~\ref{theorem: regularity of phi inverse}. 
	
	As $\pd{\phi}{y_2}$ is continuous, it follows from Corollary~\ref{corollary: uniform positive lower bound for phi x2} that \eqref{e:dy2phigeqGamma>0} holds.
	Finally, by \eqref{asymptPhi} and \eqref{identity: relation of x derivative of phi and phi inverse}, one has \eqref{eq4.28.5}.
\end{proof}

Now we are ready to prove Theorem~\ref{theorem: existence of solution without sign condition}.

\begin{proof}[Proof of Theorem~\ref{theorem: existence of solution without sign condition}]
	Proposition~\ref{theorem: existence of stream function without sign condition} guarantees the existence and uniqueness, as well as the monotonicity \eqref{e:dy2phigeqGamma>0} and asymptotic behavior \eqref{eq4.28.5}, of solutions (in $Y$) to the boundary value problem \eqref{transformed boundary value problem without extension}.  This, together with the change of variables~\eqref{change of variables}-\eqref{definition: definition of phi} and Lemma~\ref{lemma: equivalence lemma}, yields the existence and uniqueness of classical solutions to the steady hydrostatic Euler equations \eqref{equation : steady state hydrostatic Euler equation} subject to the boundary condition \eqref{boundary condition for the flow in perturbed domain}, as well as the sign condition \eqref{e:v1>0inbarOmegap}.  In order to prove Theorem \ref{theorem: existence of solution without sign condition}, it remains to verify the asymptotic behavior of the velocity field $(v_1, v_2)$. 
	
	It follows from direct computations that $\pd{\varphi}{x_1}$, $\pd{\varphi}{x_2}$, $\pd{\phi}{y_1}$ and $\pd{\phi}{y_2}$ are related by
	\begin{equation*}
		\begin{pmatrix}
			\pd{\varphi}{x_1}(x_1, x_2)\\
			\pd{\varphi}{x_2}(x_1, x_2)
		\end{pmatrix}
		=
		\begin{pmatrix}
			1 & -\frac{\left(s_0'(x_1) +  \mathfrak{s}(x_1, x_2)s'(x_1)\right)}{s(x_1)}\\
			0 & \frac{1}{s(x_1)}
		\end{pmatrix}
		\begin{pmatrix}
			\pd{\phi}{y_1}\left(x_1,  \mathfrak{s}(x_1, x_2)\right)\\
			\pd{\phi}{y_2}\left(x_1,  \mathfrak{s}(x_1, x_2)\right)
		\end{pmatrix},
	\end{equation*}
	where
	\begin{equation}
		\mathfrak{s}(x_1, x_2) :=\frac{x_2-s_0(x_1)}{s(x_1)}.
	\end{equation}
	Hence by the definition of \eqref{negative gradient per of stream function}, one has
	\begin{equation}\label{equation: explicit form of the flow in terms of phi}
		\begin{aligned}
			\qquad\begin{pmatrix}
				v_1\\
				v_2
			\end{pmatrix}
			=
			\begin{pmatrix}
				\frac{1}{s(x_1)}\pd{\phi}{y_2}\left(x_1, \mathfrak{s}(x_1,x_2)\right)\\
				-\pd{\phi}{y_1}\left(x_1, \mathfrak{s}(x_1,x_2)\right) + \frac{\left(s_0'(x_1) + \mathfrak{s}(x_1,x_2)s'(x_1)\right) \pd{\phi}{y_2}\left(x_1, \mathfrak{s}(x_1,x_2)\right)}{s(x_1)}
			\end{pmatrix}.
		\end{aligned}
	\end{equation}
	Since $\phi\in Y$, this implies that 
	\[\Psi(y_1, y_2):= \pd{\phi}{y_2}(y_1, y_2)\]
	is a function in  $C^1\left(\overline{\Om_0}\right)$.  To justify the upstream behavior~\eqref{upstreambehavior}, we notice that the uniform limit of $\phi$ satisfies \eqref{problem: ODE for the uniform limit of phi}, with $f(\varphi) = (v^-_1)'(\kappa(\varphi))$, and that $\pd{\phi}{y_2}$ converges uniformly to $(\phi^{-\infty})'$ as $y_1\to -\infty$. On the other hand, the unique solution of \eqref{problem: ODE for the uniform limit of phi} is
	\[\phi^{-\infty}(y_2)=\varphi^{-}(y_2) =\int_{0}^{y_2}v^-_1(s)\,ds.\]
	This implies that
	\[
	(\phi^{-\infty})'(y_2) = v^-_1(y_2),
	\]
	and hence, $v_1=\partial_y\phi\to v_1^-$ uniformly as $y_1\to -\infty$. Also, it follows from Assumption \eqref{structural assumption on Omega p}, Proposition~\ref{theorem: existence of stream function without sign condition}, and \eqref{equation: explicit form of the flow in terms of phi} that $v_2\to 0$ as $x_1 \to -\infty$. By a similar argument, we can justify the downstream behavior \eqref{downstreambehavior}, where $v_1^+ := \frac{(\phi^\infty)'}{\sigma}$, and $\phi^\infty$ satisfies
		\begin{equation}\label{equation: ODE of phi in downstream}
			\left\{\begin{aligned}
				(\phi^{\infty})'' = \sigma^2f(\phi^{\infty})\quad\text{in}\ (0,1)\\
				\phi^{\infty}(0) = 0,\quad\phi^{\infty}(1) = \mfc > 0
			\end{aligned}\right.
		\end{equation}
		with $f$ and $\mfc$ defined in \eqref{deff} and \eqref{mass flux constant}, respectively.
		
		This finishes the proof of Theorem  \ref{theorem: existence of solution without sign condition}.
	\end{proof}

It follows from \eqref{eq4.28.5} and \eqref{equation: explicit form of the flow in terms of phi} that the velocity has asymptotic behavior \eqref{equality: behavior of the flow at downstream for compact case}  if $\Om_p$ satisfies \eqref{generalnozzle} instead of (A3) of Assumption \ref{structural assumption on Omega p}. This is exactly what we claimed in Remark \ref{remarkdownstream}.

%\appendix
%\subfile{Appendix B}
\medskip 

\textbf{Acknowledgement.}  The work of Wong was partially supported by the HKU Seed Fund for Basic Research under the project code 201702159009, the Start-up Allowance for Croucher Award Recipients, and Hong Kong General Research Fund (GRF) grants with project numbers 17306420, 17302521, and 17315322. The work of Xie was partially supported by NSFC grant 11971307,  Natural Science Foundation of Shanghai 21ZR1433300, and Program of Shanghai Academic Research Leader 22XD1421400.
%\bibliographystyle{plain}
%\bibliography{References}

\end{document}